\documentclass[11pt, english]{article}

\usepackage[margin= 21mm,bottom=21mm, top= 20mm]{geometry}
\usepackage{amsthm}
\usepackage{amsmath}
\usepackage{amssymb}
\usepackage{setspace}
\usepackage{mathtools}
\usepackage{graphicx}
\graphicspath{ {./images/} }
\usepackage[hidelinks]{hyperref}
\usepackage{cleveref}
\usepackage{hyperref}
\usepackage{enumitem}
\usepackage{framed}
\usepackage[width=.9\textwidth]{caption}
\usepackage{xspace}
\usepackage{thmtools} 
\usepackage{thm-restate}
\usepackage{comment}

\usepackage{thmtools}
\usepackage{thm-restate}

\usepackage{floatrow}

\usepackage{tikz}
\usepackage{mathdots}
\usepackage{xcolor}
\usepackage{diagbox}
\usepackage{colortbl}
\usepackage[absolute,overlay]{textpos}

\usetikzlibrary{calc}
\usetikzlibrary{decorations.pathreplacing}
\usetikzlibrary{positioning,patterns}
\usetikzlibrary{arrows,shapes,positioning}
\usetikzlibrary{decorations.markings}

\tikzstyle{edge}=[very thick]
\definecolor{bostonuniversityred}{rgb}{0.8, 0.0, 0.0}
\definecolor{arsenic}{rgb}{0.23, 0.27, 0.29}
\tikzstyle{diredge}=[postaction={decorate,decoration={markings,
		mark=at position .56 with {\arrow[scale = 1.3, black]{stealth};}}}]
\tikzset{
    arrow/.style={decoration={markings, mark=at position 0.5 with
    {\fill(-0.18*#1,-0.06*#1) -- (0,0) -- (-0.18*#1,0.06*#1) -- cycle;}, black}, postaction={decorate}},
    arrow/.default=1
}
\tikzset{
    arow/.style={decoration={markings, mark=at position 1 with
    {\fill(-0.09*#1,-0.03*#1) -- (0,0) -- (-0.09*#1,0.03*#1) -- cycle;}}, postaction={decorate}},
    arow/.default=1
}
\tikzset{
    arrrow/.style={decoration={markings, mark=at position 0.9 with
    {\fill(-0.09*#1,-0.03*#1) -- (0,0) -- (-0.09*#1,0.03*#1) -- cycle;}}, postaction={decorate}},
    arow/.default=1
}

\newcommand{\fitellipsis}[2] % first and second node names without parentheses
{\draw [fill=white]let \p1=(#1), \p2=(#2), \n1={atan2(\y2-\y1,\x2-\x1)}, \n2={veclen(\y2-\y1,\x2-\x1)}
    in ($ (\p1)!0.5!(\p2) $) ellipse [ x radius=\n2/2+0cm, y radius=1.1cm, rotate=\n1];
}
\newcommand{\Fitellipsis}[2] % first and second node names without parentheses
{\draw [fill=white]let \p1=(#1), \p2=(#2), \n1={atan2(\y2-\y1,\x2-\x1)}, \n2={veclen(\y2-\y1,\x2-\x1)}
    in ($ (\p1)!0.5!(\p2) $) ellipse [ x radius=\n2/2+0cm, y radius=1.4cm, rotate=\n1];
}

\theoremstyle{plain}

\newtheorem*{thm*}{Theorem}
\newtheorem{thm}{Theorem}[section]
\Crefname{thm}{Theorem}{Theorems}

\newtheorem*{lem*}{Lemma}
\newtheorem{lem}[thm]{Lemma}
\Crefname{lem}{Lemma}{Lemmas}

\newtheorem*{claim*}{Claim}
\newtheorem{claim}{Claim}[]
\crefname{claim}{Claim}{Claims}
\Crefname{claim}{Claim}{Claims}

\Crefname{prop}{Proposition}{Propositions}

\Crefname{remar}{Remark}{Remarks}

\crefname{cor}{Corollary}{Corollaries}

\newtheorem*{conj*}{Conjecture}
\newtheorem{conj}[thm]{Conjecture}
\crefname{conj}{Conjecture}{Conjectures}

\Crefname{qn}{Question}{Questions}

\newtheorem*{obs*}{Observation}
\newtheorem{obs}[thm]{Observation}
\Crefname{obs}{Observation}{Observations}

\Crefname{ex}{Example}{Examples}

\theoremstyle{definition}

\Crefname{prob}{Problem}{Problems}

\newtheorem{defn}[thm]{Definition}
\Crefname{defn}{Definition}{Definitions}

\theoremstyle{remark}

\renewenvironment{proof}[1][]{\begin{trivlist}
\item[\hspace{\labelsep}{\bf\noindent Proof#1.\/}] }{\qed\end{trivlist}}

\newcommand{\remove}[1]{}

\newcommand{\F}{\mathcal{F}}
\newcommand{\HH}{\mathcal{H}}
\newcommand{\inter}{\mathrm{int}}
\newcommand{\seg}{\mathrm{seg}}
\newcommand{\lex}{\mathrm{lex}}
\newcommand{\End}{\mathrm{End}}

\title{\vspace{-1 cm}
Hamiltonicity of expanders: optimal bounds and applications}
\date{}

\author{
Nemanja Dragani\'c\thanks{
Mathematical Institute, University of Oxford, UK. Research supported by SNSF project 217926.\\
\emph{Email}: \textbf{nemanja.draganic@maths.ox.ac.uk}}
\and 
Richard Montgomery \thanks{Mathematics Institute, University of Warwick, Coventry, CV4 7AL, UK. Research supported
by the European Research Council (ERC) under the European Union Horizon 2020 research and innovation programme (grant
agreement No. 947978). \emph{Email}: \textbf{richard.montgomery@warwick.ac.uk.}}
\and 
David Munh\'a Correia\thanks{
Department of Mathematics, ETH, Z\"urich, Switzerland. Research supported in part by SNSF grant 200021\_196965.
\newline
\emph{Emails}: \textbf{\{david.munhacanascorreia, benjamin.sudakov\}@math.ethz.ch}.
}
\and
Alexey Pokrovskiy\thanks{Department of Mathematics, University College London, Gower Street, London WC1E 6BT, UK. \\ \emph{Email}: \textbf{dralexeypokrovskiy@gmail.com.}}
\and
Benny Sudakov\footnotemark[3]}

\begin{document} 
\maketitle
\begin{abstract}
An $n$-vertex graph $G$ is a $C$-expander if $|N(X)|\geq C|X|$ for every $X\subseteq V(G)$ with $|X|< n/2C$ and there is an edge between every two disjoint sets of at least $n/2C$ vertices.
We show that there is some constant $C>0$ for which every $C$-expander is Hamiltonian. In particular, this implies the well known conjecture of Krivelevich and Sudakov from 2003 on Hamilton cycles in $(n,d,\lambda)$-graphs. This completes a long line of research on the Hamiltonicity of sparse graphs, and has many applications.
\end{abstract}
\section{Introduction}

A Hamilton cycle in a graph $G$ is a cycle that contains all the vertices of $G$. The presence of such a cycle categorizes $G$ as Hamiltonian. This fundamental concept in Graph Theory has been extensively studied, for example see~\cite{ajtai1985first,bollobas1987algorithm,chvatal1972note,MR3545109,cuckler2009hamiltonian,ferber2018counting,hefetz2009hamilton,krivelevich2011critical,krivelevich2014robust,kuhn2014proof,kuhn2013hamilton, posa:76} and the surveys~\cite{gould2014recent, MR3727617}. Deciding whether a graph is Hamiltonian or not is an NP-complete problem, and thus it is an important area of research to find simple conditions which imply Hamiltonicity. One classic example is Dirac's theorem \cite{dirac1952some} that any graph with $n\geq 3$ vertices and minimum degree at least $n/2$ is Hamiltonian. Another beautiful condition, by Chv\'atal and Erd\H{o}s \cite{chvatal1972note}, is that if the connectivity of a graph is at least its independence number, then the graph is Hamiltonian. Other famous Hamiltonicity conditions include those of Chv\'atal \cite{chvatal1972hamilton}, Jackson \cite{jackson1980hamilton} and Nash-Williams \cite{nash1971edge}. Most known conditions for Hamiltonicity, however, require the graph to be very dense. All the results mentioned above imply that the graph has linear minimum degree, except the Chv\'atal-Erd\H{o}s result, which still implies minimum degree $\Omega(\sqrt{n})$. Hence, it is of particular interest to find Hamiltonicity conditions which also apply to sparse graphs.

A key area of research towards this in the last 50 years has been on Hamiltonicity in sparse random graphs, and, in particular, which binomial random graphs $G(n,p)$
and which random regular graphs $G_{n,d}$ are likely to be Hamiltonian. P\'osa~\cite{posa:76} was the first to determine the threshold for Hamiltonicity in $G(n,p)$, introducing the famous rotation-extension technique which quickly become a widely used tool with innumerable applications (including in this paper). After a refinement of P\'osa's result by Korshunov \cite{korshunov}, in 1983 Bollobás \cite{bollobas1984evolution} and Komlós and Szemerédi~\cite{komlos1983limit}
independently showed that, if $p = (\log n + \log \log n + \omega(1))/n$, then $G(n, p)$
is almost certainly Hamiltonian. As is well-known, when $p = (\log n + \log \log n - \omega(1))/n$, $G(n,p)$ almost certainly has a vertex with degree 1, and hence no Hamilton cycle.
Random regular graphs, then, may be far sparser yet plausibly Hamiltonian with high probability, and, after significant focus on the problem, it is now known that $G_{n,d}$ will almost surely have a Hamilton cycle for all $3 \leq d \leq n - 1$. For further details on this, see the work of Cooper, Frieze, and Reed \cite{cooper2002random} and Krivelevich, Sudakov, Vu, and Wormald \cite{krivelevich2001random}.

The well-established understanding of Hamiltonicity in random graphs presents an important step towards the search for simple properties of sparse graphs which imply Hamiltonicity. It points to considering natural `pseudorandom' conditions which are required by a deterministic graph to resemble a random graph. However, forgoing the randomness of $G(n,p)$ and relying only on these pseudorandom properties to find a Hamilton cycle presents a significantly firmer challenge, similar to the generalisation of other problems from random to pseudorandom graphs.
Pseudorandom graphs have been systematically studied since work by Thomason~\cite{thomason1987pseudo,thomason1987random} in the 1980's, a history that can be found in the survey by Krivelevich and Sudakov~\cite{krivelevich2006pseudo}. The most studied class of pseudorandom \emph{regular} graphs was introduced by Alon and is defined using spectral properties. Recalling that if a graph $G$ is $d$-regular then its largest eigenvalue is $d$, we denote the second largest eigenvalue of $G$ in absolute value by $\lambda(G)$. Then, a graph $G$ is an \emph{$(n,d,\lambda)$-graph} if it is $d$-regular with $n$ vertices and satisfies $|\lambda(G)|\leq \lambda$.

The first major step towards understanding the Hamiltonicity of such pseudorandom graphs was made by Krivelevich and Sudakov in 2003 in their influential paper \cite{KS:03}. They showed that if $d$ is sufficiently larger than $\lambda$, then the graph is Hamiltonian. More precisely,
\begin{equation}\label{eqn:KSbound}
\frac{d}{\lambda} \geq 1000 \cdot  \frac{\log n \cdot (\log \log \log n)}{ (\log \log n)^2}
\end{equation}
implies that every $(n,d,\lambda)$-graph is Hamiltonian. In the same paper, Krivelevich and Sudakov made the beautiful conjecture that \eqref{eqn:KSbound} can be replaced by $\frac{d}{\lambda}\geq C$ for some large constant $C$, as follows.

\begin{conj}\label{conj:ks}
    There exists $C>0$ such that if $\frac{d}{\lambda}\geq C$, then every $(n,d,\lambda)$-graph is Hamiltonian.
\end{conj}

\noindent That is, if the absolute value of every other eigenvalue of a regular graph is at most a small constant fraction of the largest eigenvalue, then the graph should be Hamiltonian.

Considering the result in~\cite{KS:03} in the context of random graphs allows us to benchmark, broadly, this progress. That is to say, the result of~\cite{KS:03} is strong enough to prove the likely Hamiltonicity of the random regular graph $G_{n,d}$ when $d\geq \log^{2-o(1)}n$, while the ultimate goal, Conjecture~\ref{conj:ks}, would be  strong enough to prove the likely Hamiltonicity of the random regular graph $G_{n,d}$ when  $d$ is at least some large constant.
Despite a great deal of attention, for example seen by the various relaxations and generalisations of the problem studied in \cite{allen2017powers,brandt2006global,hefetz2009hamilton,krivelevich2006pseudo,krivelevich2011number}, and many incentivising applications (see Section~\ref{sec:applications}), the bound at \eqref{eqn:KSbound} established in \cite{KS:03} remained unchallenged for 20 years. Only recently,  Glock, Munh\'a Correia and Sudakov~\cite{glock2023hamilton} finally improved this, significantly strengthening the result of~\cite{KS:03} by showing that, for some large constant $C>0$, $d/\lambda\geq C\log^{1/3}n$ suffices to imply Hamiltonicity. Moreover, they showed that Conjecture~\ref{conj:ks} holds in the special case where $d\geq n^\alpha$, for any fixed $\alpha$, that is, in this case \eqref{eqn:KSbound} can be weakened to $d/\lambda\geq C$.

Krivelevich and Sudakov applied their bound at~(\ref{eqn:KSbound}) to other problems on Hamiltonicity in sparse graphs; these and other applications are discussed in Section~\ref{sec:applications}. To allow applications to non-regular graphs, other pseudorandom conditions which imply Hamiltonicity were also studied. Motivated by this, shortly after Conjecture~\ref{conj:ks} was stated, several papers considered an even stronger conjecture, singling out the key properties of $(n,d,\lambda)$-graphs thought to give some potential for proving Hamiltonicity. To state this even stronger conjecture, whose variant appeared for example in \cite{brandt2006global}, we need the following definition.

\begin{defn}\label{def:expander}
An $n$-vertex graph $G$ with $n \geq 3$ is a \textit{$C$-expander} if,
\begin{itemize}
\item[(a)] $|N(X)|\geq C|X|$ for all vertex sets $X\subseteq V(G)$ with $|X|< n/2C$, and,
\item[(b)] there is an edge between any disjoint vertex sets $X,Y\subseteq V(G)$ with $|X|,|Y|\geq n/2C$.
\end{itemize}
\end{defn}

\begin{conj}\label{conj:main}
    For every sufficiently large $C>0$, every $C$-expander is Hamiltonian.
\end{conj}

\noindent In 2012, Hefetz, Krivelevich and Szab\'o~\cite{hefetz2009hamilton} made progress on this problem; the precise expansion conditions used in their result can be found in Theorem 1.1 of~\cite{hefetz2009hamilton} (in particular weakening (a) in our Definition~\ref{def:expander}), but imply that every $(\log^{1-o(1)}{n})$-expander is Hamiltonian.

In this paper, we prove Conjecture~\ref{conj:main}, thus completing an extensive line of research on Hamiltonicity problems.
\begin{thm}\label{thm:main}
    For every sufficiently large $C>0$, every $C$-expander is Hamiltonian.
\end{thm}
\noindent This result has a large number of applications, as discussed below. In particular, it is a standard exercise to show that for every $C>0$ there exists a constant $C_0$ such that, if $\frac{d}{\lambda}\geq C_0$, then every $(n,d,\lambda)$-graph is a $C$-expander. Thus, clearly \Cref{conj:ks} is implied by \Cref{thm:main}, giving the following.
\begin{thm}\label{thm:ndlambda}
     There is a constant $C>0$ such that if $\frac{d}{\lambda}\geq C$ then every $(n,d,\lambda)$-graph is Hamiltonian.
\end{thm}

\noindent To prove Theorem~\ref{thm:main}, we take a very different approach to the previous work on this problem, and so give a detailed outline of the proof in Section~\ref{sec:sketch}. 
We note here that in fact a stronger result than Theorem~\ref{thm:main} holds -- such graphs are not only Hamiltonian, but Hamilton-connected. Furthermore, the Hamilton cycle in \Cref{thm:ndlambda} can be found in polynomial time. We discuss these strengthenings in the concluding remarks.
We finish this section by giving some examples of the applications of Theorems~\ref{thm:main} and~\ref{thm:ndlambda}.

\subsection{Applications}\label{sec:applications}

Pseudorandom conditions for Hamiltonicity have found a large variety of applications, and Theorems~\ref{thm:main} and \ref{thm:ndlambda} immediately improve the bounds required for many such applications. We will discuss here three applications: Hamiltonicity in random Cayley graphs, Hamiltonicity in well-connected graphs and Hamilton cycles with few colours in edge-coloured graphs. Further applications, many of them discussed in \cite{glock2023hamilton,KS:03}, can be found in  various other fields, ranging from problems in positional games (see, e.g., \cite{ferber2015biased,clemens2012fast, hefetz2009hamilton,hefetz2014positional}), to questions about finding coverings and packings of Hamilton cycles in random and pseudorandom graphs (see, e.g., \cite{ferber2019packing,draganic2023optimal,hefetz2014optimal}), Hamiltonicity thresholds in different random graph models (see, e.g., \cite{frieze2018trace,alon2008k,frieze2008logconcave}), and for various other problems (see, e.g., \cite{hefetz2012sharp,johannsen2013expanders}), including as far afield as Alon and Bourgain's work on additive patterns in multiplicative subgroups~\cite{AB:14}.

\medskip

\noindent\textbf{Hamiltonicity in random Cayley graphs.} In 1969, Lov\'asz~\cite{lovasz1969combinatorial} made the following famous conjecture about the Hamiltonicity of vertex-transitive graphs, which are graphs in which any vertex can be mapped to any other vertex by an automorphism.
\begin{conj}\label{conj:lovasz}
Every connected vertex-transitive graph contains a Hamilton path, and, except for five known examples, a Hamilton cycle.
\end{conj}
\noindent As Cayley graphs are vertex-transitive and none of the five known exceptions in Conjecture~\ref{conj:lovasz} are Cayley graphs, Lov\'asz's conjecture implies the following earlier conjecture, posed in 1959, by Strasser \cite{rapaportstrasser:59}.
\begin{conj}\label{conj:transitive}
Every connected Cayley graph is Hamiltonian.
\end{conj}
\noindent Conjecture~\ref{conj:transitive} is known to be true when the underlying group is abelian, but the only progress towards the conjectures in general is a result of Babai \cite{babai1979long} that every vertex-transitive $n$-vertex graph contains a cycle of length $\Omega(\sqrt{n})$ (see~\cite{deVos} for a recent improvement by DeVos) and a result of Christofides, Hladk\'y and M\'ath\'e \cite{christofides2014hamilton} that every vertex-transitive graph of linear minimum degree contains a Hamilton cycle.

The ``random version'' of Conjecture~\ref{conj:transitive} is a natural relaxation of the original problem. Alon and Roichman~\cite{AR:94} showed that there is a constant $C>0$ for which, for any group $G$, the Cayley graph generated by a random set $S$ of $C \log |G|$ elements, $\Gamma(G,S)$ say, is almost surely connected. Hence, an important instance of Conjecture \ref{conj:transitive} is to show that $\Gamma(G,S)$ is almost surely Hamiltonian. This problem was also stated as a conjecture by Pak and Radoi\v{c}i\'{c}~\cite{PR:09}. Theorem \ref{thm:ndlambda} resolves this conjecture. Indeed, Alon and Roichman \cite{AR:94} showed that if $|S| \geq C \log |G|$ for some large constant $C$, then $\Gamma(G,S)$ is almost surely an $(n,d,\lambda)$-graph with $d/ \lambda \geq K$ for some large constant $K$. Thus, we obtain the following.

\begin{thm}\label{thm:cayley random}
Let $C$ be a sufficiently large constant. Let $G$ be a group of order $n$ and $d \geq C\log n$. If $S\subseteq G$ is a set of size $d$ chosen uniformly at random, then, with high probability, $\Gamma(G,S)$ is Hamiltonian.
\end{thm}

\smallskip

\noindent\textbf{Hamiltonicity in well-connected graphs.}
For any function $f : \mathbb{Z}^+ \rightarrow
\mathbb{R}$, we say that a graph $G$ is \emph{$f$-connected} if $|A \cap B| \geq f(\min(|A \setminus B|,|B \setminus A|))$ for every two subsets $A,B \setminus V(G)$ such that $V(G) = A \cup B$ and there is no edge between $A \setminus B$ and $B \setminus A$. In~\cite{brandt2006global}, Brandt, Broersma, Diestel, and Kriesell proved that if $f(k) \geq
2(k+1)^2$ for every $k \in \mathbb{N}$ then $G$ is Hamiltonian. This bound was improved to $k\log k + O(1)$ in \cite{hefetz2009hamilton}. Brandt et al.\ also conjectured that there exists a
function $f$ which is linear in $k$ yet ensures
Hamiltonicity. It is not difficult to check that if $f(k) = Ck$ for all $k$, then an $f$-connected graph is a $(C/2)$-expander. Indeed, given a set $X \subseteq V(G)$ of size at most $n/C$ and its neighbourhood, $N(X)$, applying the $f$-connected condition with $A = X \cup N(X), B = V(G) \setminus X$ shows that $|N(X)| \geq C|X|/2$, while, given two disjoint sets $X,Y \subseteq V(G)$ of size at least $n/C$, setting $A = V(G) \setminus X, B = V(G) \setminus Y$ and applying the $f$-connected condition shows that there is an edge between $X,Y$. Therefore, Theorem \ref{thm:main} implies the conjecture, as follows.
\begin{thm} \label{fconnected} Let $C$ be a sufficiently large constant and let $f(k)=Ck$ for every $k\in \mathbb{Z}^+$. Then, any $f$-connected graph with at least 3 vertices is Hamiltonian.
\end{thm}

\smallskip

\noindent\textbf{Hamilton cycles using few colours.}
An \emph{optimally} edge-coloured graph $G$ is one which is properly coloured using the fewest possible number of colours. Akbari, Etesami, Mahini, and
Mahmoody \cite{AEMM:07} proved that any optimally coloured $n$-vertex complete graph $K_n$ has a Hamilton cycle containing edges of at most $8\sqrt{n}$ colours, and conjectured there should always be such a cycle using only $O(\log n)$ colours, which would be best possible up to a multiplicative constant.
The bound from \cite{AEMM:07} on the number of colours was later improved to $O(\log^3 n)$ by Balla, Pokrovskiy and Sudakov~\cite{BPS:17}. Their strategy is to randomly pick $d=O(\log^3 n)$ colours and show that the subgraph of these colours is an $(n,d,\lambda)$-graph with high probability, for some appropriate $\lambda$, and apply the result from \cite{KS:03} on the Hamiltonicity of $(n,d,\lambda)$-graphs. Using their improved condition, Glock, Munh\'a Correia and Sudakov~\cite{glock2023hamilton} showed that this is possible with only $O(\log^{5/3} n)$ colours using the same method. Applying instead the bound in Theorem~\ref{thm:ndlambda} immediately proves the conjecture (see \cite{glock2023hamilton} for more details).
\begin{thm}\label{thm:Ham colors}
Every optimally coloured $K_n$ has a Hamilton cycle with $O(\log n)$ colours.
\end{thm}

\section{Preliminaries}\label{sec:preliminaries}
\subsection{Notation.} 
Our notation is standard, and we recall here only the most important. A graph $G$ has vertex $V(G)$ and edge set $E(G)$, and we set $|G|=|V(G)|$. A \emph{linear forest} is a graph consisting of a collection of vertex disjoint paths. We denote by $\text{End}(\F)$ the set of vertices which are endpoints of paths in $\F$.
Two vertices are \emph{at distance $\ell$ in $G$} if the length of the shortest path between them in $G$ is $\ell$. 
For a subset of vertices $X$, we denote by $G[X]$ the induced subgraph of $G$ with vertex set $X$, by $\Gamma_G(X)$ the set of vertices adjacent to at least one vertex in $X$, and by $N_G(X)$ the (outer) neighbourhood $N_G(X)=\Gamma_G(X)\setminus X$. We omit $G$ in the subscript when it is clear from the context which graph we are working with. 

\subsection{Proof outline}\label{sec:sketch}
Where $G$ is an $n$-vertex $C$-expander, for some large constant $C$, we wish to find a Hamilton cycle in $G$. Before describing our methods, it is instructive to briefly recall the well-known P\'osa rotation-extension approach, and how the approach works on random graphs but not on pseudorandom graphs (for a more detailed approach than this sketch, see, for example,~\cite{bollobas1984evolution}).

Take a maximal length path $P$ in $G$, with endvertices $x$ and $x_0$, say. If $x$ has a neighbour $y$ in $G$ on $P$ whose neighbour on $P$ which is nearest to $x$,  $z$ say, is not equal to $x$, then the path $P$ can be rotated by adding $xy$ and removing $yz$ (see Figure~\ref{fig:rotation}). 
\begin{figure}[t]
\begin{center}
\begin{tikzpicture}[scale=1.4]

  % Define vertices
  \coordinate (x) at (0,0);
  \coordinate (x_0) at (4,0);
  \coordinate (y) at (3,0);
  \coordinate (z) at (2.75,0); % Slightly to the left and above of z

 \draw [line width=4pt, color=blue, opacity=0.3] (x)-- (z);
    \draw [line width=4pt, color=blue, opacity=0.3] (y)-- (x_0);
  \draw [line width=4pt, color=blue, opacity=0.3] (x)to[ bend left=30] (y);
  
  % Draw horizontal segment between x and y
  \draw (x) -- (x_0) node[midway, below, thick] {};

  % Draw slightly curved edge from x to z
  \draw [densely dashed,black!70] (x) to[bend left=30] (y) node[midway, above] {};
     
  % Place vertices
  \foreach \point in {x_0, x, y, z}
    \fill (\point) circle (1pt) node[below] {$\point$};

    \draw ($(x)-(0.4,0)$) node {$P$\,:};

\end{tikzpicture}\;\;\;\;\;\;
\begin{tikzpicture}[scale=1.4]

  % Define vertices
  \coordinate (x) at (0,0);
  \foreach \num/\xarg in {1/0.5,2/0.75,3/1.75,4/2,5/3.5,6/3.75}
  {
  \coordinate (x\num) at (\xarg,0);
  }
  \coordinate (x_0) at (4,0);
  \coordinate (y) at (3,0);
  \coordinate (z) at (2.75,0); % Slightly to the left and above of z

 %\draw [line width=4pt, color=blue, opacity=0.3] (x)-- (z);
   % \draw [line width=4pt, color=blue, opacity=0.3] (y)-- (x_0);
  \draw [line width=4pt, color=blue, opacity=0.3] (x)to[ bend left=30] (y);
  
  % Draw horizontal segment between x and y
  \foreach \vxx/\vxy in {x/x1,x2/x3,x6/x_0,x4/z,y/x5}
  {
 \draw [line width=4pt, color=blue, opacity=0.3] (\vxx)-- (\vxy);
  \draw (\vxx) -- (\vxy) node[midway, below, thick] {};
  }
   \draw (z) -- (y) node[midway, below, thick] {};

  % Draw slightly curved edge from x to z
  \draw [densely dashed,black!70] (x) to[bend left=30] (y) node[midway, above] {};
     
  % Place vertices
  \foreach \point in {x, y, z}%x1,x2,x3,x4,x5,x6,
    \fill (\point) circle (1pt) node[below] {$\point$};

  \foreach \point in {x1,x2,x3,x4,x5,x6,x_0},
    \fill (\point) circle (1pt);

    \draw ($(x)-(0.4,0)$) node {$\F$\,:};

\end{tikzpicture}
\vspace{-0.2cm}
\caption{On the left, a rotation of an $xx_0$-path $P$ by removing $zy$ and adding $xy$ to get an $xx_0$-path highlighted in blue. On the right, an example rotation of a linear forest $\F$ to get the 4 paths highlighted in blue, thus removing $x$ from the set of endvertices and adding $z$.} 
\label{fig:rotation}
\end{center}
\end{figure}
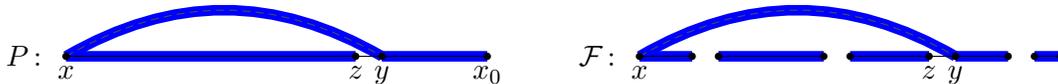
This gives us a path with the same length as $P$, where one endvertex is $x_0$ and the other is a new endvertex, $z$.
As P\'osa showed, if $G$ is a $C$-expander, then this can be done iteratively to show that there are at least (say) $n/3$ different new endvertices $v$, say those in $E_{x_0}$, so that you can rotate $P$ repeatedly to get an $x_0v$-path $P_v$ with the same length as $P$. Then, each of these paths $P_v$ can be rotated without changing the endvertex $v$, to get at least $n/3$ new endvertices, say those in $F_v$. If $G$ contains any edge in $\{vu:v\in E_{x_0},u\in F_v\}$, $G$ has a cycle with length $|P|$, whence, as $G$ is connected (a trivial consequence of the expansion condition), if $|P|<n$ then we can find a longer path than $P$, a contradiction. Thus, $G$ has a Hamilton cycle. Of course, $G$ may not contain any edge in $\{vu:v\in E_{x_0},u\in F_v\}$. When working with random graphs, however, we can reserve some random edges and sprinkle them in to find an edge in  $\{vu:v\in E_{x_0},u\in F_v\}$ and hence a longer path than $P$ -- doing this iteratively can get a Hamilton path and thus a Hamilton cycle.

In pseudorandom graphs we cannot do this sprinkling. Thus, the dream strategy would be to find two large sets $A,B$ with $A\times B\subset \{vu:v\in E_{x_0},u\in F_v\}$, so that we can apply Definition~\ref{def:expander}b) to find the edge. Indeed, for their result discussed in the introduction, Hefetz, Krivelevich and Szab\'o~\cite{hefetz2009hamilton} found such sets $A,B$ with $|A|,|B|\geq n/\log^{1-o(1)}n$. However, large linear sets $A$ and $B$ (as would be required for Theorem~\ref{thm:main}) seem too hard to find, essentially because the rotations performed around each endvertex of the path interfere with each other by rotating sections of the paths.

For our approach, instead of rotating only paths, we perform rotations of disjoint union of paths, i.e., of linear forests. The details of these rotations can be found in Section~\ref{sec:rotations}, but, analogously, we alter a linear forest and preserve all but one endvertex of these paths (one sample rotation is depicted in Figure~\ref{fig:rotation}), iteratively developing many different possibilities for the endvertex we are changing. The implementation is much more intricate, but, essentially, if most of the vertices in a linear forest $\mathcal{F}$ are not in long paths then we can take two endvertices $x$ and $x'$ of different paths, divide the paths in $\mathcal{F}$ into two groups $\mathcal{F}_x$ and $\mathcal{F}_{x'}$, and perform rotations just within these groups to find new endvertices in place of $x$ and $x'$. As these two sets of rotations are done using disjoint linear forests, they can be done independently, to get sets of new endvertices $A$ and $B$ respectively, such that any edge from $A$ to $B$ will allow us to connect two paths together, reducing the number of paths in the linear forest by 1 and losing the endvertices $x$ and $x'$. This is done in Section~\ref{sec:rotations}.

We then wish to apply this iteratively to reduce the number of paths in our linear forest. However these rotations may create long paths (so that it is hard to find the partition into forests $\mathcal{F}_x$ and $\mathcal{F}_{x'}$). Thus, we combine this with an argument showing if $\mathcal{F}$ has too many long paths then we can replace it with a linear forest $\mathcal{F}'$ whose lengths are more evenly spread yet whose endvertices are still a subset of the endvertices of $\mathcal{F}$. This is also done using rotations, and is carried out in Section~\ref{sec:shortpaths} (which includes a sketch of the proof once the necessary definitions have been introduced). Starting with a spanning linear forest $\mathcal{F}$, we can alternately perform rotations and add an edge (reducing the number of paths by one) while losing any two endvertices (where, moreover, with some more effort, we can specify these two endvertices) and then replace the current linear forest with one with more evenly spread path lengths (so that the first step keeps working). Ultimately, this allows us to reach $\mathcal{F}'$, a spanning linear forest of at most $n^{0.8}$ paths, all of whose endvertices were also endvertices of paths in $\mathcal{F}$.

This argument can be pushed further to end up with fewer paths, but cannot produce a Hamilton path let alone a Hamilton cycle. Thus, we need some way to connect the paths we produce into a Hamilton cycle. We do this by setting aside at the outset an \emph{$(A,B)$-linking structure $H$} in $G$, where $|A|=|B|$ and this has the property that $H$ contains a spanning linear forest connecting any desired partition of pairs $a,b$ with $a\in A$ and $b\in B$. More precisely, this is defined as follows.

\begin{defn}
A graph $H$ with disjoint sets $A,B \subseteq V(H)$ of equal size is said to be an $(A,B)$-linking structure if for every bijection $\varphi: A\rightarrow B$ there exist vertex disjoint paths $P_1, \ldots, P_{|A|}$ of equal length such that the following hold.
\begin{enumerate}
    \item The paths cover all the vertices of $H$, so that $V(H)=V(P_1) \cup \ldots \cup V(P_{|A|})$.
    \item For each $i\in [|A|]$, the path $P_i$ has endpoints $a$ and $\varphi(a)$ for some $a\in A$.
\end{enumerate}
\end{defn}

\noindent By using properties of sorting networks, adapting an approach of Hyde, Morrison, M\"uyesser, and Pavez-Sign\'e~\cite{hyde2023spanning},  we can find an $(A,B)$-linking structure $H$ in $G$ with $|A|=|B|=n^{0.9}$ and $|H|=o(n^{0.95})$. Taking then an initial linear forest $\mathcal{F}$ which covers $G-V(H)\setminus (A\cup B)$ and has every vertex in $A\cup B$ among its endvertices, we then apply our methods to reach a linear forest $\mathcal{F}'$ with the same vertex set as $\mathcal{F}$, but whose endvertices are exactly those in $A\cup B$ (with no isolated vertices in $\F'$). Then, applying the linking property of $H$ allows us to
connect the paths in $\mathcal{F}''$ together into our desired Hamilton cycle (see Figure~\ref{fig:hamcycle}).

\vspace{-0.2cm}

\begin{figure}[htb]
\centering
\begin{tikzpicture}[scale=0.6]
% Define the distance between nodes
%\begin{scope}[scale=0.5]
\def\nodeDist{1cm}
%rectangle
\draw[rounded corners, line width=1pt] (-0.3,-6.35*\nodeDist) rectangle (4.3,0.35*\nodeDist);

\draw[dotted,thick] (0.3, -6.35*\nodeDist) -- (0.3, 0.35*\nodeDist);
\draw[dotted,thick] (3.7, -6.35*\nodeDist) -- (3.7, 0.35*\nodeDist);

% Vertical boxes around the vertices
%\draw[dotted, rounded corners, very thick] (-0.7, -6.5*\nodeDist) rectangle +(1, 7*\nodeDist);
%\draw[dotted, rounded corners, very thick] (3.7, -6.5*\nodeDist) rectangle +(1, 7*\nodeDist);

%red
\foreach \y in {0,...,2} {
\draw[red,thick] plot[smooth] coordinates {(0,0-\y*\nodeDist) (-0.5-0.7*\y,1+0.25*\y) (4.5+0.7*\y,1+0.25*\y)(4,0-\y*\nodeDist)};
}

% Draw the red half circles to the left for a4a5 and a6a7
\foreach \y in {3,5} {
\draw[red,thick] plot[smooth] coordinates {(0,0-\y*\nodeDist) (-3,-\y*\nodeDist+0.3*\nodeDist) (-3,-\y*\nodeDist-1.3*\nodeDist)(0,0-\y*\nodeDist-\nodeDist)};
}

\foreach \y in {3,5} {
\draw[red,thick] plot[smooth] coordinates {(4,0-\y*\nodeDist) (7,-\y*\nodeDist+0.3*\nodeDist) (7,-\y*\nodeDist-1.3*\nodeDist)(4,0-\y*\nodeDist-\nodeDist)};
}

% Draw the blue edges between b\i and a\{i+1}
\foreach \i in {0,...,5} {
  \draw[thick, blue] (4,-\i*\nodeDist) -- (0,-\i*\nodeDist-\nodeDist);
}
% The last connection from b7 to a1
\draw[thick,blue] (4,-6*\nodeDist) -- (0,0);
%\end{scope}

% Draw the vertices a1 to a7
\foreach \y in {0,...,6} {
  \filldraw (0,-\y*\nodeDist) circle (3pt);
}
%X
\node[color=black, scale=1]  at (2,-6*\nodeDist) {$X$};
\node[color=black, scale=1]  at (4,-6.75*\nodeDist) {$B$};
\node[color=black, scale=1]  at (0,-6.75*\nodeDist) {$A$};

% Draw the vertices b1 to b7
\foreach \y in {0,...,6} {
  \filldraw (4,-\y*\nodeDist) circle (3pt);
}
\end{tikzpicture}
\caption{Given a linear forest $\F'$ (depicted in red) which spans $G-(X\setminus (A\cup B))$ with no isolated vertices and endvertices $\End(\F') = A \cup B$, if $G[X]$ is an $(A,B)$-linking structure, then the paths in blue can be found disjointly while using all the vertices in $X$, thus linking $\F'$ into a Hamilton cycle.} 
\label{fig:hamcycle}
\end{figure}
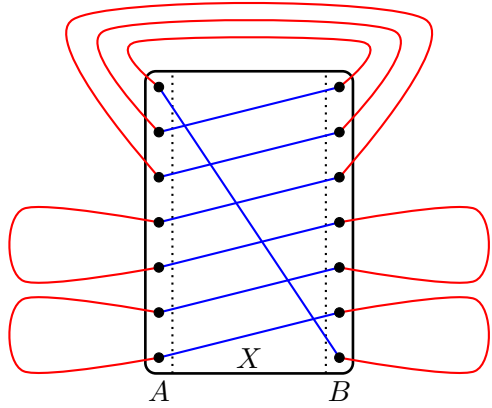

%%%%%%%%%%%%%%%%%%%%%%%%%%%%
%%%%%%%%%%%%%%%%%%%%%%%%%%%%

\subsection{Expansion and linear forests}
\noindent In order to carry out P\'osa rotations in linear forests, we need to define the interior of sets with respect to the linear forest, and a variation of expansion where we only consider the neighbourhood of sets within the interior of the linear forest. For this, we will use the following definitions.% of interior.
\begin{defn}
Let $G$ be a graph, and $U,V\subseteq V(G)$ with $|U|\geq 10^5C$. We say that  $U$ \textit{$C$-expands into $V$ (in $G$)} if all $X\subseteq U$ with $|X|\leq |U|/5000C$ have $|N_G(X)\cap V|\geq C|X|$.
\end{defn}
\begin{defn}
Let $G$ be a graph and $\F$ a linear forest in $G$. For a set $U \subseteq V(\F)$, we define the \emph{interior of $U$ in $\F$} to be the set $\inter_{\F}(U):=\{u\in U: N_{\F}(u)\subseteq U\}$. Furthermore, we say that $U$ is an \textit{$(\F,C)$-expander} if $U$ $C$-expands into $\inter_\F(U)$.
\end{defn}
\noindent The following observation is immediate from the previous definition.
\begin{obs}\label{obs:interior}
For all linear forests $\F$ and subsets $X, Y \subseteq V(\F)$ it must be that $\inter_{\F}(X \setminus Y) \geq \inter_{\F}(X) - 3|Y|$.
\end{obs}

\noindent To carry out rotations within only some of the paths in a linear forest, we will need to find expanding subgraphs sitting within the interior of these specific paths. Often expander subgraphs can be found using only Definition~\ref{def:expander}b) by removing a maximal set of vertices which does not expand and is not too large (see, for example,~\cite[Section 3.7]{montgomery2019spanning}). Here, we work similarly, but remove multiple sets $B_{i,j}$ to find a collection of 4 sets $U_i'$ with expansion properties (with respect to some linear forest $\mathcal{F}$) as follows. 

\begin{lem}\label{lem:forestexpander} Let $C\geq 10^5$ and $C'=C/5000$.
Let $G$ be a $C$-expander containing a linear forest $\F$, and let $U_1, U_2, U_3, U_4\subset V(G)$ be disjoint sets with $|\inter_{\F}(U_i)|\geq n/500$ for each $i\in [4]$. Then, there are subsets $U_i'\subseteq U_i$, $i\in [4]$, such that 
\begin{itemize}
    \item for each $i,j\in [4]$, $U_i'$ $C'$-expands into $\inter_{\F}(U_j')$, and 
    \item $|U_i'|\geq |U_i|-2n/C$ for each $i\in [4]$.     
\end{itemize}
In particular, for each $i\in [4]$, $U'_i$ is an $(\F,C')$-expander.
\end{lem}
\begin{proof}
Pick sets $B_{i,j}\subseteq U_i$, $i,j\in [4]$, satisfying the following:
\begin{enumerate}[label = (\arabic{enumi})]
\item $|B_{i,j}|\leq 2n/C$ for all $i,j\in [4]$.
\item $\left|N(B_{i,j})\cap \inter_{\F}\left(U_j \setminus \bigcup_{t=1}^{4}B_{j,t} \right) \right|\leq C'|B_{i,j}|$ for all $i,j\in [4]$.
\item $\sum_{i,j\in [4]} |B_{i,j}|$ is as large as possible, subject to (1) and (2).
\end{enumerate}
First note that this is indeed possible, as the sets $B_{i,j}=\emptyset$, $i,j\in [4]$, satisfy (1) and (2).
Let $B:=\bigcup_{i,j\in [4]} B_{i,j}$.
\begin{claim}\label{clm:Bijsmall}
For each $i,j\in [4]$, $|B_{i,j}|< n/2C$.
\end{claim}
\begin{proof}
Suppose otherwise. Then, the definition of a $C$-expander implies that $B_{i,j}$ is adjacent to all but at most $n/2C$ vertices of $V(G)\setminus B_{i,j}$ and, thus, to all but at most $n/2C$ vertices in $\inter_{\F}(U_j\setminus B)$. Hence, 
\begin{align*}    
|N(B_{i,j}) \cap \inter_{\F}(U_j\setminus B)|&\geq |\inter_{\F}(U_j\setminus B)| - n/2C \geq |\inter_{\F}(U_j)| - 3|B \cap U_j|- n/2C \\&\geq |\inter_{\F}(U_j)| - 3\cdot 16 \cdot 2n/C - n/2C \geq (10C'-96.5)n/C > C'|B_{i,j}|,
\end{align*}
contradicting (2). Note that in the previous inequalities we used Observation \ref{obs:interior}, and that $|\inter_{\F}(U_i)|\geq n/500$ for each $i\in [4]$ and $C'=C/5000$.
\renewcommand{\qedsymbol}{$\boxdot$}
\end{proof}
\renewcommand{\qedsymbol}{$\square$}
\noindent We now show that setting $U'_i := U_i \setminus B$ for each $i\in [4]$ gives the first item of the desired outcome of the lemma.
\begin{claim}
For each $i,j\in [4]$, $U_i\setminus B$ $C'$-expands into $\inter_{\F}(U_j\setminus B)$.
\end{claim}
\begin{proof}
Suppose for contradiction that there is some $i,j\in [4]$ for which there is some $X\subseteq U_i\setminus B$ with $|N(X)\cap \inter_{\F}(U_j\setminus B)|<C'|X|$ and $|X|\leq n/5000C'\leq  n/C$. Then, $|B_{i,j}\cup X| \leq |B_{i,j}| + |X| \leq 2n/C$ by Claim~\ref{clm:Bijsmall} and, furthermore,
\begin{align*}
|N(B_{i,j}\cup X)\cap \inter_{\F}(U_j\setminus B)|&\leq |N(B_{i,j})\cap \inter_{\F}(U_j\setminus B)|+|N(X)\cap \inter_{\F}(U_j\setminus B)|\\
&\leq C'|B_{i,j}| + C'|X| = C'|B_{i,j} \cup X|.
\end{align*}
For each $s,t\in [4]$, we have  
$|N(B_{s,t})\cap \inter_{\F}\left(U_t \setminus(B\cup X)\right)|\leq |N(B_{s,t})\cap \inter_{\F}(U_t\setminus B)|\leq C' |B_{s,t}|$.
Thus, replacing $B_{i,j}$ by $B_{i,j}\cup X$ in $\{B_{s,t}:s,t\in [4]\}$ gives a family of sets satisfying (1) and (2), but with a larger total number of vertices, contradicting (3).
\renewcommand{\qedsymbol}{$\boxdot$}
\end{proof}
\renewcommand{\qedsymbol}{$\square$}
Finally, setting $U_i'=U_i\setminus B$ for each $i\in [4]$ satisfies the lemma. Indeed, note that since $|B \cap U_i| \leq 4 \cdot n/2C$ for each $i\in [4]$ by Claim~\ref{clm:Bijsmall}, we have that $|U'_i| \geq |U_i| - 2n/C$.
\end{proof}
%%%%%%%%%%%%%%%%%%%%%%%%%%%%%%5
%%%%%%%%%%%%%%%%%%%%%%%%%%%%%%%%5
\section{Rotations of linear forests}\label{sec:rotations}
In this section we give our methods for reducing the number of paths in a spanning linear forest in an expander, by rotating the linear forest.
After defining our rotations and making some simple observations in Section~\ref{sec:rot:simple}, we will prove some intermediary results for these rotations in Section~\ref{sec:rot:lesssimple} before proving the key result of this section, Lemma~\ref{lem:generalpathforestrot}, in Section~\ref{sec:mainrotation}.

\subsection{Rotation definition and simple observations}\label{sec:rot:simple}
 We first define rotations of linear forests, which can be compared to rotations of paths by lining up the paths of the linear forest as in Figure~\ref{fig:rotation}. Rotations of linear forests can, however, look quite different depending on whether the edge used to rotate is between two different paths, or between vertices in the same path, and on whether these paths are isolated vertices, or not; the different possiblities are illustrated in Figure~\ref{fig:rot}.
\begin{defn}
    Let $P_1$ and $P_2$ be two (possibly equal) paths of a linear forest $\F$ in a graph $G$. Let $x$ be an endpoint of $P_1$ and let $z$ be a neighbour (in $G$) of $x$ in $P_2$. Let $y$ be a vertex adjacent to $z$ in $P_2$ if $|P_2|\geq 1$ (if $P_1=P_2$, then let $y$ be the vertex closer to $x$ in $P_1$, where we may have $y=x$), and let $y=z$ otherwise. Then, the linear forest $\F'$ obtained from $\F$ by removing the edge $yz$ (if $y\neq z$) and adding the edge $xz\in E(G)$ is called a \emph{$1$-rotation of $\F$ in $G$}.

 We call the vertex $x$ the \emph{old endpoint}, $y$ the \emph{new endpoint} (indeed, note that $y \in \End(\F')$), and $z$ the \emph{pivot}. The edge $zy$ is called the \emph{broken edge} of the rotation. 
We will refer to the process of removing $yz$ (if it exists) and adding $xz$ to $\F$ as a 1-rotation, as well as the resulting linear forest $\F'$, for example, saying that, by performing a 1-rotation on $\F$ with old endpoint $x$, pivot $z$ and new endpoint $y$, we get the 1-rotation $\F'$ of $\F$.
\end{defn}
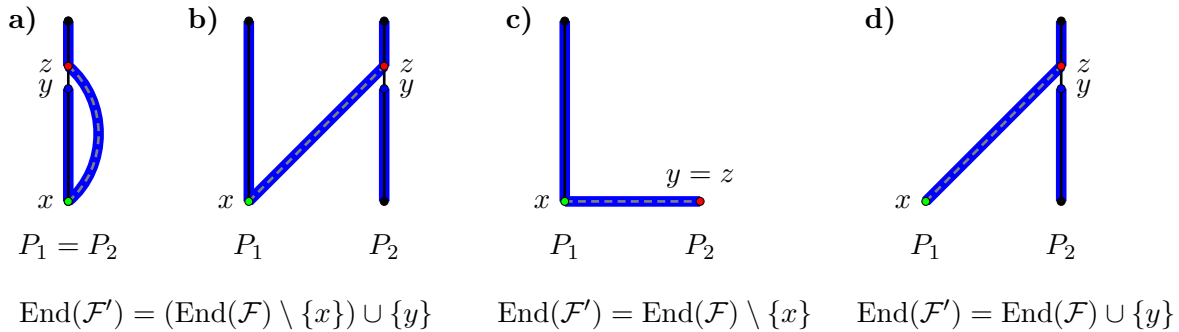
\begin{figure}[h]
\begin{center}
\begin{tikzpicture}[scale = 0.6]
\begin{scope}[xshift=1cm]
\draw [line width=4pt, color=blue, opacity=0.3] (-2,0)-- (1,3);
\draw [line width=4pt, color=blue, opacity=0.3] (1,4)-- (1,3);
\draw [line width=4pt, color=blue, opacity=0.3] (-2,4)-- (-2,0);
\draw [line width=4pt, color=blue, opacity=0.3] (1,2.5)-- (1,0);
\node[color=black]  at (-2.5,0) {$x$};
\node[color=black]  at (-2,-1) {$P_1$};
\node[color=black]  at (1,-1) {$P_2$};
\draw [fill=black] (-2,4) circle (2.5pt);
\draw [line width=1pt] (-2,0)-- (-2,4);
%\node[color=black]  at (-0.8,1) {$P_1$};
\draw [fill=black] (1,0) circle (2.5pt);
\draw [fill=black] (1,4) circle (2.5pt);
\node[color=black]  at (-3,4) {\textbf{b)}};
\draw [line width=1pt] (1,0)-- (1,2.5);
\draw [line width=1pt] (1,4)-- (1,3);
%\node[color=black]  at (3.8,1) {$P_2$};
\draw [line width=1pt,densely dashed,black!50] (-2,0)-- (1,3);
\draw [line width=1pt] (1,3) -- (1,2.5);
\draw [fill=red] (1,3) circle (2.5pt);
\draw [fill=blue] (1,2.5) circle (2.5pt);
\draw [fill=green] (-2,0) circle (2.5pt);
\node[color=black]  at (1.5,3) {$z$};
\node[color=black]  at (1.5,2.5) {$y$};

\end{scope}

\draw [line width=4pt, color=blue, opacity=0.3] (14,0)-- (17,3);
\draw [line width=4pt, color=blue, opacity=0.3] (17,4)-- (17,3);
%\draw [line width=4pt, color=green, opacity=0.5] (6,4)-- (6,0);
\draw [line width=4pt, color=blue, opacity=0.3] (17,2.5)-- (17,0);
\node[color=black]  at (13.5,0) {$x$};
\node[color=black]  at (14,-1) {$P_1$};
\node[color=black]  at (17,-1) {$P_2$};
%\draw [fill=black] (6,4) circle (2.5pt);
%\draw [line width=1pt] (6,0)-- (6,4);
%\node[color=black]  at (-0.8,1) {$P_1$};
\draw [fill=black] (17,0) circle (2.5pt);
\draw [fill=black] (17,4) circle (2.5pt);
\node[color=black]  at (13,4) {\textbf{d)}};
\draw [line width=1pt] (17,0)-- (17,4);
%\node[color=black]  at (3.8,1) {$P_2$};
\draw [line width=1pt,densely dashed,black!50] (14,0)-- (17,3);
\draw [fill=red] (17,3) circle (2.5pt);
\draw [fill=blue] (17,2.5) circle (2.5pt);
\draw [fill=green] (14,0) circle (2.5pt);
\node[color=black]  at (17.5,3) {$z$};
\node[color=black]  at (17.5,2.5) {$y$};

\draw [line width=4pt, color=blue, opacity=0.3] (6,0)-- (9,0);
%\draw [line width=4pt, color=green, opacity=0.5] (15,4)-- (15,3);
\draw [line width=4pt, color=blue, opacity=0.3] (6,4)-- (6,0);
%\draw [line width=4pt, color=blue, opacity=0.5] (15,2.5)-- (15,0);
\node[color=black]  at (5.5,0) {$x$};
\node[color=black]  at (6,-1) {$P_1$};
\node[color=black]  at (9,-1) {$P_2$};
\draw [fill=black] (6,4) circle (2.5pt);
\node[color=black]  at (5,4) {\textbf{c)}};
\draw [line width=1pt] (6,0)-- (6,4);
%\node[color=black]  at (-0.8,1) {$P_1$};
%\draw [fill=black] (15,0) circle (2.5pt);
%\draw [fill=black] (15,4) circle (2.5pt);
%\draw [line width=1pt] (15,0)-- (15,4);
%\node[color=black]  at (3.8,1) {$P_2$};
\draw [line width=1pt,densely dashed,black!50] (6,0)-- (9,0);
\draw [fill=red] (9,0) circle (2.5pt);
%\draw [fill=blue] (15,2.5) circle (2.5pt);
\draw [fill=green] (6,0) circle (2.5pt);
\node[color=black]  at (9,0.5) {$y = z$};
%\node[color=black]  at (15.5,2.5) {$z$};

\draw [line width=4pt, color=blue, opacity=0.3] (-5,0) to [bend right = 50] (-5,3);
\draw [line width=4pt, color=blue, opacity=0.3] (-5,4)-- (-5,3);
%\draw [line width=4pt, color=green, opacity=0.5] (-3,4)-- (-3,0);
\draw [line width=4pt, color=blue, opacity=0.3] (-5,2.5)-- (-5,0);
\node[color=black]  at (-6,4) {\textbf{a)}};
\node[color=black]  at (-5.5,0) {$x$};
\node[color=black]  at (-5,-1) {$P_1=P_2$};
\draw [fill=black] (-5,4) circle (2.5pt);
\draw [line width=1pt] (-5,0) -- (-5,4);
%\node[color=black]  at (-0.8,1) {$P_1$};
\draw [fill=black] (-5,0) circle (2.5pt);
\draw [fill=black] (-5,4) circle (2.5pt);

%\node[color=black]  at (3.8,1) {$P_2$};
\draw [line width=1pt,densely dashed,black!50] (-5,0) to [bend right = 50] (-5,3);
\draw [fill=red] (-5,3) circle (2.5pt);
\draw [fill=blue] (-5,2.5) circle (2.5pt);
\draw [fill=green] (-5,0) circle (2.5pt);
\node[color=black]  at (-5.5,3) {$z$};
\node[color=black]  at (-5.5,2.5) {$y$};

\node[color=black]  at (-1.5,-2.5) {$\End  (\mathcal{F}') = \left( \End  (\mathcal{F}) \setminus \{x\} \right) \cup \{y\} $};

\node[color=black]  at (8,-2.5) {$\End  (\mathcal{F}') =  \End  (\mathcal{F}) \setminus \{x\}$};
%\node[color=black]  at (7.5,-3) {$z \in \End  (\mathcal{F}') \cap \End  (\mathcal{F})$};

\node[color=black]  at (16,-2.5) {$\End  (\mathcal{F}') =  \End  (\mathcal{F}) \cup \{y\} $};
%\node[color=black]  at (15.5,-3) {$x \in \End  (\mathcal{F}') \cap \End  (\mathcal{F})$};

\end{tikzpicture}
\caption{Different 1-rotations of $\F=\{P_1,P_2\}$ to get $\F'$ by removing the edge $yz$ (if it exists) and adding the dashed edge $xz$, with the new paths highlighted in blue and the change in endvertices given underneath. In \textbf{a)} $P_1 = P_2$, and in \textbf{b)--d)} $P_1\neq P_2$. In \textbf{b)} neither $x$ or $z$ are isolated, in \textbf{c)} $y=z$ is isolated but not $x$ and in \textbf{d)} $x$ is isolated but not $z$.} 

\vspace{-0.3cm}

\label{fig:rot}
\end{center}
\end{figure}

\noindent As in P\'osa's original rotation-extension method on paths, we will repeatedly rotate linear forests to find many different new endvertices. For this, we will now define a \emph{$k$-rotation}.
\begin{defn}
A \emph{$k$-rotation} of a forest $\F$ in a graph $G$ is any forest which is obtained by a sequence of $k$ consecutive $1$-rotations starting with $\F$ in $G$, where the old endpoint of each $1$-rotation is the new endpoint of the previous one and the pivot is at least at distance $3$ in $\F$ from the old endpoint of the first $1$-rotation and the pivot in the $i$-th rotation for all $i<k$ (and thus the broken edge belongs to $\F$ as well). We say that the $k$-rotation has \emph{old endpoint} $x$, or \emph{starts at $x$}, and \emph{new endpoint} $y$ if $x$ is the old endpoint of the first $1$-rotation and $y$ is the new endpoint of the last $1$-rotation.

For a set $U\subseteq V(G)$ we denote by $E^k_U(v,\F)$ the set of new endpoints obtained by $i$-rotations of $\F$ starting from $v$ (i.e., with old endpoint $v$) with $i\leq k$, such that the pivots of each $1$-rotation are contained in $\inter_\F(U)$, so that each broken edge belongs to $\F[U]$. We call such rotations \emph{$(U,i)$-rotations}, and where $i$ is not specified call them \emph{$U$-rotations}. 
\end{defn}

\noindent Next we prove some simple results on $k$-rotations.
\begin{lem}\label{obs:endpoint change}
Let $\F'$ be a $(U,k)$-rotation of $\F$, with $k\geq 1$, starting with vertex $v$ and let $u$ be the new endpoint.  Then, the following holds.
\begin{enumerate}[label=\emph{(\Alph*)}]
    \item\label{obs:number of pivots} If $E^k_U(v,\F)\subseteq S\subseteq E^{k+1}_U(v,\F)$, then the set of pivots used in the $1$-rotations to get new endpoints in $S$ is of size at most $2|S|$. 
\end{enumerate}
\noindent Furthermore, if $\F$ has no isolated vertices, then the following also hold.
\begin{enumerate}[label=\emph{(\Alph*)}]\addtocounter{enumi}{1}
    
    \item\label{obs: new endpoint is maybe iso} The only possible isolated vertex in $\F'$ is $u$, which furthermore can only be isolated in $\F'$ if $u \in \text{End}(\F)$.

    \item\label{obs: endpoints} $\text{End}(\F') = \left(\text{End}(\F)\cup\{u\} \right) \setminus\{v\}$.
\end{enumerate}
\end{lem}

\begin{proof}
    For~\emph{\ref{obs:number of pivots}} note that every pivot used for a $i$-rotation in $S$ is adjacent to a vertex in $S$ in the forest $\F$, or is equal to a vertex in $S$ in the case this pivot is isolated in $\F$. This is because all broken edges in a $k$-rotation $\F'$ are always edges in the initial forest $\F$. As the maximum degree of the linear forest $\F$ is $2$, there are at most $2|S|$ pivots used.

    We prove part \emph{\ref{obs: new endpoint is maybe iso}} by induction on $k$ and under the weaker assumption that $\F$ has no isolated vertices \emph{except for, possibly, $v$}. For $k=1$, let $z$ be the pivot of the rotation, and note that $v$ is adjacent to $z$ in $\F'$, hence it is not isolated in $\F'$. Since $\F$ contains no other isolated vertices, note that the only possible isolated vertices in $\F'$ are $z$ and $u$, as $zu$ is the only deleted edge in $\F$. But $z$ is adjacent to $v$ in $\F'$, so only $u$ can be isolated in $\F$, and this happens if and only if it did not have any other adjacent vertex in $\F$ besides $z$, i.e.\ if and only if it is an endpoint of $\F$. 
    Suppose now the statement holds for $k-1\geq 1$. Consider a $k$-rotation $\F_k$ of $\F$ with old endpoint $v$ and new endpoint $u$. By definition of $k$-rotation, we get a $(k-1)$-rotation $\F_{k-1}$ of $\F$ with old endpoint $v$ and new endpoint $w$, such that $\F_k$ is a $1$-rotation of $\F_{k-1}$ with old endpoint $w$ and new endpoint $u$ and a pivot $z$. By the induction hypothesis, the only possible isolated vertex in $\F_{k-1}$ is $w$. 
    Applying the induction hypothesis to $\F_{k-1}$ and then to its $1$-rotation $\F_k$ implies that $u$ is the only possible isolated vertex in $\F_k$, and that this is the case if and only if $u$ is an endpoint of $\F_{k-1}$. As $z$ is at distance at most $3$ from all the previous pivots (used for $\F_{k-1}$), we know that $u$ is an endpoint of $\F_{k-1}$ if and only if it is an endpoint of $\F$, which completes this part of the proof.

    For part~\emph{\ref{obs: endpoints}}, we also proceed by induction.
    Let $k=1$ and let $z$ be the pivot of the $1$-rotation $\F'$. If $z$ is the neighbouring vertex of $v$ in $\F$, the statement trivially holds as $\F=\F'$ and $u=v$.    
    Otherwise, note that $v$ has precisely one neighbour in $\F$, while it has another neighbour $z$ in $\F'$, so $v$ is not an endpoint in $\F'$.
    On the other hand, the edge $zu$ is removed from $\F$, so $u$ has at most one neighbour in $\F'$, so it is an endpoint.
    Assume the statement holds for $k-1\geq 1$. Consider a $k$-rotation $\F_k$ of $\F$ with old endpoint $v$ and new endpoint $u$. As before, we get a $(k-1)$-rotation $\F_{k-1}$ of $\F$ with old endpoint $v$, new endpoint $w$, such that $\F_k$ is a $1$-rotation of $\F_{k-1}$ with old endpoint $w$ and new endpoint $u$ and a pivot $z$. By the induction hypothesis, we know that $\text{End}(\F_{k-1})=(\text{End}(\F)\cup\{w\})\setminus \{v\}$. We distinguish two cases. If $w$ is not an endpoint of $\F$, then by~\emph{\ref{obs: new endpoint is maybe iso}} it is not isolated in $\F_{k-1}$, and hence by applying the induction hypothesis to the $1$-rotation $\F_k$ of $\F_{k-1}$, we see that $\text{End}(\F_{k})=(\text{End}(\F_{k-1})\cup\{u\})\setminus \{w\}=(\text{End}(\F)\cup\{u\})\setminus \{v\}$. In the second case, when $w$ is an endpoint of $\F$, again by~\emph{\ref{obs: new endpoint is maybe iso}} we have that $w$ is the only isolated vertex in $\F_{k-1}$. Thus, the $1$-rotation $\F_k$ of $\F_{k-1}$ still contains $w$ as an endpoint, while similarly to the case $k=1$ the only new created endpoint is $u$. Hence $\text{End}(\F_{k})=((\text{End}(\F)\cup\{w\})\setminus \{v\})\cup \{u\}=(\text{End}(\F)\setminus\{v\})\cup \{u\}$ as required.
\end{proof}

\subsection{Intermediate results on rotations}\label{sec:rot:lesssimple}

As mentioned in the proof sketch, a key part of P\'osa's rotation-extension technique is to show that performing rotations iteratively in an expander creates a linear-sized set of potential endvertices. We now show the corresponding result for our linear forest rotations.
\begin{lem}\label{Lem_endpoints_in_expander}
Let $C>100$ and let $\F$ be a linear forest in a graph $G$, $v$ an endpoint of $\F$,  and $U$ an $(\F,C)$-expander containing $v$. Then $|E^{k}_{U}(v,\F)|\geq |U|/10^5$ for $k=2\log_C n$.
\end{lem}
\begin{proof}
    Suppose, to the contrary, that $|E^{k}_{U}(v,\F)|< |U|/10^5$.  For each $i\in [k-1]$, we have $|E^{i}_{U}(v,\F)|\leq |E^{k}_{U}(v,\F)|< |U|/10^5$.
        Now,  
    for each $i\in [k-1]$, and each pair $(x,y)$ with $x\in E^i_U(v,\F), y \in \text{int}_{\F}(U)$ and $xy\in E(G)$, $y$ is a good candidate to use as a pivot for a 1-rotation of a linear forest which has $x$ as an endvertex, created by rotating $\F$ up to $k-1$ times. However, by our definition, we need to have that this pivot is at least at distance $3$ in $\F$ from all the previous pivots in the same rotations. In total, though, as we have  at most $2|E^{i}_{U}(v,\F)|$ pivots used by part~\emph{\ref{obs:number of pivots}} of \Cref{obs:endpoint change}, this will rule out at most  $10|E^{i}_{U}(v,\F)|$ potential pivots $y$.   
Finally, we note that two pivots in $N(E^{i}_{U}(v,\F)) \cap \text{int}_{\F}(U)$ may possibly correspond to the same new endpoint. Therefore, we have 
\begin{equation}\label{eqn:Ek1}
|E^{i+1}_U(v,\F)|\geq \frac{|N(E^{i}_{U}(v,\F)) \cap \text{int}_{\F}(U)|-10|E^{i}_{U}(v,\F)|}{2}.
\end{equation}
If $|E^{i}_{U}(v,\F)|\geq |U|/5000C$, then, by considering a subset $S\subseteq E^{i}_{U}(v,\F)$ of size $|U|/5000C$, we have that since $U$ is an $(\F,C)$-expander, $|N(S) \cap \text{int}_{\F}(U)| \geq C|S|$ and so, $|N(E^{i}_{U}(v,\F)) \cap \text{int}_{\F}(U)|\geq C|S|-|E^{i}_{U}(v,\F)|\geq |U|/5000-|U|/10^5\geq 9|U|/10^5 + 10|E^{i}_{U}(v,\F)|$, where we are using that $|E^{i}_{U}(v,\F)| \leq |U|/10^5$. Then, by \eqref{eqn:Ek1}, $|E^{i+1}_U(v,\F)|\geq |U|/10^5$, a contradiction. Therefore, $|E^{i}_{U}(v,\F)|\leq |U|/5000C$, whence we have $|N(E^{i}_{U}(v,\F)) \cap \text{int}_{\F}(U)|\geq C|E^{i}_{U}(v,\F)|$ as $U$ is an $(\F,C)$-expander, so that \eqref{eqn:Ek1} implies
\[
|E^{i+1}_U(v,\F)|\geq \frac{C|E^{i}_{U}(v,\F)|}{3}.
\]
As this holds for every $0\leq i\leq k-1$, we have $|E^k_U(v,\F)|\geq (C/3)^k>n$, a contradiction.
\end{proof}

\noindent The next lemma essentially says that if a linear forest $\F$ in an $n$-vertex expander can be split into two so that rotations can be done on each subcollection of paths independently starting from two vertices $x$ and $y$ respectively, then we can alter $O(\log n)$ edges in $\F$ to decrease the number of paths by 1 while losing exactly $x$ and $y$ as endvertices, as follows.

\begin{lem}\label{lem:rotationforclosing}
Let $C>C'>10^{10}$ and let $G$ be an $n$-vertex $C$-expander and $\mathcal{F}$ a spanning linear forest in $G$. Suppose $\F$ has no isolated vertices and let $x,y$ be two endpoints of $\F$. Let $X,Y$ be $(\mathcal{F},C')$-expanders of size at least $0.0001n$ such that $x \in X,y \in Y$, no path in $\mathcal{F}$ intersects both $X$ and $Y$, and there is no vertex in $(X\setminus \{x\})\cup (Y\setminus \{y\})$ which is an endpoint in $\F$. Then, there is a spanning linear forest $\mathcal{F'}$ in $G$ for which $|E(\F)\Delta E(\F')|=O(\log n)$, $\End (\mathcal{F'}) = \End (\mathcal{F}) \setminus \{x,y\}$, and $\F'$ contains no isolated vertices.
\end{lem}
\begin{proof}
Let $\mathcal{F}_X$ be the collection of paths in $\mathcal{F}$ which intersect $X$ and $\mathcal{F}_Y$ be the collection of paths in $\mathcal{F}$ which intersect $Y$, so that $\mathcal{F}_X \cap \mathcal{F}_Y = \emptyset$. By Lemma~\ref{Lem_endpoints_in_expander}, there exists a subset $X' \subseteq X$ of size at least $|X|/10^5 > n/C$ such that for all $x' \in X'$ there is an $X$-rotation $\mathcal{F}_{x'}$ of $\mathcal{F}$ with new endpoint $x'$. Note crucially that since $\mathcal{F}_{x'}$ is an $X$-rotation, we have that $\mathcal{F}_Y \subseteq \mathcal{F}_{x'}$. Similarly, there is a subset $Y' \subseteq Y$ of size at least $|Y|/10^5 > n/C$ such that for all $y' \in Y'$ there is a $Y$-rotation $\mathcal{F}_{y'}$ of $\mathcal{F}$ with new endpoint $y'$. Let $x'y'$ be an edge between $X'$ and $Y'$, which exists as $G$ is a $C$-expander. 

Now, the forest $\mathcal{F}_{x'}$ is such that $\mathcal{F}_{Y} \subseteq \mathcal{F}_{x'}$. Therefore, any $Y$-rotation of $\mathcal{F}$ is identical when restricted to $Y$ to any $Y$-rotation of $\mathcal{F}_{x'}$ with the same change in endpoints. Therefore, since $y'$ is a new endpoint created by a $Y$-rotation $\mathcal{F}_{y'}$ of $\mathcal{F}$, there is then a linear forest $\mathcal{F}_{x',y'}$ which is obtained by performing the $X$-rotation
and $Y$-rotation consecutively. Note that $\mathcal{F}_{x',y'}$ has no isolated vertices by \Cref{obs:endpoint change}~\emph{\ref{obs: new endpoint is maybe iso}}. In particular, then, $x',y'$ are not isolated in $\mathcal{F}_{x',y'}$ and so adding the edge $x'y'$ to  $\mathcal{F}_{x',y'}$ we obtain the forest $\F'$ which has neither $x'$ or $y'$ as endvertices. More precisely, $\End (\mathcal{F'}) = \End (\mathcal{F}) \setminus \{x,y\}$. Since $\F_x$ and $\F_y$ are $O(\log n)$-rotations, we conclude that $\F$ and $\F'$ differ in at most $O(\log n)$ edges, i.e., $|E(\F)\Delta E(\F')|=O(\log n)$.
\end{proof}

\noindent Unfortunately, even when we can split our linear forest into two subcollections of paths with many vertices, we cannot guarantee that we can rotate within them independently (as for Lemma~\ref{lem:rotationforclosing}) from two arbitrary endvertices $x$ and $y$. Therefore, in our proof, we will first need to rotate in the whole linear forest to replace $x$ and $y$ by vertices for which we can do this. This is carried out in Section~\ref{sec:mainrotation}, using the following lemma.

\begin{lem}\label{lem:mainrotation}
Let $C\geq C'>10^6$ and let $G$ be an $n$-vertex $C$-expander which contains a linear forest $\F$ with no isolated vertices. Let $U,V\subseteq V(G)$ be two subsets of vertices. Let $u\in U$ be an endpoint of $\F$ and $U$ an $(\F,C')$-expander with $|U|\geq 10^5n/C$. Let $V$ satisfy $|\inter_\F(V)|\geq 11n/C$ and contain no endpoints of $\F$.
Then, there is an $(U \cup V,O(\log n))$-rotation $\F'$ of $\F$ and $v \in V$ such that $ \End \left(\mathcal{F}' \right)= ( \End  \left(\mathcal{F} \right) \setminus $  $ \{u\}) \cup \{v\}$.
Furthermore, we have that $\mathcal{F}'[V \setminus v] = \mathcal{F}[V \setminus v]$, all the edges broken in the successive $1$-rotations except the last one are not in $\F[V]$, and $\F'$ has no isolated vertices.
\end{lem}

\begin{proof}
First apply \Cref{Lem_endpoints_in_expander} to see that $|E^{2\log_{C'} n}_{U}(u,\F)|\geq |U|/10^5\geq n/C$. We consider two cases.

\vspace{0.3cm}
\noindent \textbf{Case 1: $E^{2\log_{C'} n}_{U}(u,\F) \cap V \neq \emptyset$.}
\vspace{0.3cm}

\noindent Let $k \leq 2\log_{C'} n$ be the smallest integer such that $E^{k}_{U}(u,\F) \cap V \neq \emptyset$ and $v \in E^{k}_{U}(u,\F) \cap V$. 
Let $\F'$ be the rotation corresponding to  the endpoint $v$.
 Then $\F'$ is a $O(\log n)$-rotation of $\F$. Furthermore, by the minimality of $k$, none of the new endpoints, except for $v$, in the successive $1$-rotations creating $\F'$ belong to $V$. Hence, none of the broken edges in these successive rotations are contained in $V$. Therefore, the only difference between $\F'[V]$ and $\F[V]$ occurs in the last $1$-rotation, with a broken edge incident to $v$. Hence, $\F'[V \setminus v] = \F[V \setminus v]$, as desired. It is easy to check that $\F'$ has no isolated vertices. Indeed, note that by \Cref{obs:endpoint change}~\emph{\ref{obs: new endpoint is maybe iso}}, the only candidate is possibly $v$, but since $V$ has no endpoints in $\F$, by \Cref{obs:endpoint change}~\emph{\ref{obs: new endpoint is maybe iso}} we are done.
 %and for it to be isolated in $\F'$, $v$ had to be an endpoint of the $k-1$ rotation obtained before performing the last $1$-rotation. This contradicts the minimality of $k$.

\vspace{0.3cm}
\noindent \textbf{Case 2: $E^{2\log_{C'} n}_{U}(u,\F) \cap V = \emptyset$.}
\vspace{0.3cm}

\noindent Let $k \leq 2\log_{C'} n$ be the smallest integer such that $|E^{k}_{U}(u,\F)|\geq n/C$, and let $E^{k-1}_{U}(u,\F)\subseteq X\subseteq E^{k}_{U}(u,\F)$ satisfy $|X|=n/C$. 
Let $P$ be the set consisting of $u$ and all the pivots used in $1$-rotations creating endpoints in $X$ and note that by part~\emph{\ref{obs:number of pivots}} of \Cref{obs:endpoint change} we have $|P|\leq 2|X|+1$. Let $P^2$ be set of vertices which are at distance at most $2$ from $P$ in the forest $\F$. Now we have that $|\text{int}_{\F}(V)\setminus P^2|\geq 11n/C-5|P| \geq 11n/C-10n/C-5\geq n/2C$. 
Since $G$ is a $C$-expander, there is an edge $xy$ with $x \in X$ and $y \in \text{int}_{\F}(V)\setminus P^2$. Consider the $(U,i)$-rotation of $\F$ for which $x$ is the last new endpoint. Combined with the $1$-rotation with old endpoint $x$, pivot $y$ and a new endpoint $v\in V$, we obtain a new rotation which satisfies two properties: its last new endpoint is in $V$, and all the previous pivots except maybe the last one are in $U$. Denote this rotation by $\F'$.
    
The above argument shows that the obtained $i$-rotation $\F'$ satisfies $i\leq 2\log_{C'} n+1<\log n$. Note that, by construction, all the pivots used to obtain $\F'$ are at least at distance $3$ in $\F$. Also note that since $\F'$ is a combination of a $(U,i)$-rotation and another $1$-rotation, such that the $(U,i)$-rotation uses no new endpoints in $V$, the only edge which is possibly broken in $\F[V]$ is in the last $1$-rotation when the edge $yv$ is broken. This immediately implies $\F'[V \setminus v]=\F[V \setminus v]$. Finally, again by \Cref{obs:endpoint change}~\emph{\ref{obs: new endpoint is maybe iso}} there are no isolated vertices in $\F'$.
\end{proof}

%%%%%%%%%%%%%%%%%%%%%%%%%%%%%%%%%
%%%%%%%%%%%%%%%%%%%%%%%%%%%%5

\subsection{The main rotation lemma}\label{sec:mainrotation}

The next lemma is one of our key results. As it is the most technical part of our proof, we will now outline its proof briefly. Given an $n$-vertex $C$-expander $G$ containing a spanning linear forest, with most of its vertices on paths that are not too long, in which $x$ and $y$ are endvertices, we wish to perform rotations on $\mathcal{F}$ until we have enough vertices potentially replacing $x$ and $y$ as endpoints that can use Definition~\ref{def:expander}b) to connect two paths so that we end up with one fewer path while losing exactly $x$ and $y$ as endpoints (see Lemma~\ref{lem:generalpathforestrot}).
Using Lemma~\ref{lem:mainrotation}, we can perform rotations on $\F$ to replace $x$ and $y$ each with linearly many different new endpoints. However, we want to do this independently, so that Definition~\ref{def:expander}b) can be applied. We manage this in two stages. First, we use rotations to replace $x$ and $y$ by two new endpoints $x_1$ and $x_2$ so that we can use rotations on two different sets of vertices ($U_1$ and $U_2$ respectively) to change $x_1$ or to change $x_2$. We then use this property to replace $x_1$ and $x_2$ by two new endpoints $x_3$ and $x_4$, so that these endpoints can be rotated independently by using vertices on two different sets of paths (by rotating with vertex sets $U_3$ and $U_4$,  whose vertices never appear together on the same path). The two sets of endpoints created by doing this will have some edge between them by Definition~\ref{def:expander}b), allowing us to join two paths together. However, this might create a cycle if $x_3$ and $x_4$ were endpoints of the same path. Therefore, instead, we first rotate to replace $x_1$ by some other vertex $h$ (lying in a fifth set $U_{\text{hop}}$) before performing one extra rotation to change $h$ to a new endpoint, $x_3$, where this brief `hop' allows us to ensure that $x_3$ is then not an endpoint of the same path as $x_4$.

\begin{lem}\label{lem:generalpathforestrot}
Let $C>10^{10}$, let $G$ be an $n$-vertex $C$-expander and let $\mathcal{F}$ be a spanning linear forest in $G$ with no isolated vertices and such that at least $0.1n$ vertices of $G$ belong to paths in $\F$ of lengths between $100$ and $\sqrt{n}$. Let also $x,y \in \End(\F)$. Then, there is a spanning linear forest $\mathcal{F}'$ in $G$ with no isolated vertices such that  $|E(\mathcal{F}) \Delta E(\mathcal{F}')| = O(\log n)$ and $\End (\mathcal{F}') = \End (\mathcal{F}) \setminus \{x,y\}$.
\end{lem}

\begin{proof}
Let $P_1, \ldots, P_t$ be the paths in $\mathcal{F}$ of length between $100$ and $\sqrt{n}$, so that $\sum_{i\in [t]} |P_i| \geq 0.1n$. Let $\mathcal{H} := \{P_1', \ldots, P_t'\}$, such that each path $P_i'$ is obtained from $P_i$ by removing the two endpoints of the path $P_i$, so that now $\sum_{i\in [t]} |P_i'| \geq 0.1n(1-\frac{1}{50})\geq 0.098n$. Take a partition of this linear forest $\mathcal H$ into $\mathcal{H}_1, \mathcal{H}_2, \mathcal{H}_3 , \mathcal{H}_4, \mathcal{H}_{\text{hop}}$, each of which span at least $0.015n$ vertices. Hence, the interior of each of these forests (since they consist of paths of length at least 98) is at least of size $\frac{96}{98}\cdot 0.015n\geq 0.01n$.

Let $C'=C/5000$. By \Cref{lem:forestexpander}, we can find subsets $U_1\subseteq V(\HH_1)$ and $U_2\subseteq V(\HH_2)$, where $U_i$ is such that it $C'$-expands into $\inter_\HH(U_j) \subseteq \inter_\F(U_j)$ for all $i,j\in\{1,2\}$. In particular, $U_1$ and $U_2$ are $(\F,C')$-expanders. Furthermore, every vertex in $U_1$ has at least $C'>100$ neighbours in $\inter_{\F}(U_2)$, and every vertex in $U_2$ has at least $C'>100$ neighbours in $\inter_{\F}(U_1)$. We also can get from \Cref{lem:forestexpander} that, for each $i\in\{1,2\}$, it holds that $|\inter_{\F}(U_i)|\geq |V(\mathcal H_i)|-2n/C\geq 0.01n -2n/C\geq 0.005n$.

We now use Lemma \ref{lem:mainrotation} to prove that the following holds.

\begin{claim} 
There is a linear forest $\mathcal{F}'$ and $x_1 \in U_1, x_2 \in U_2$ such that $\End (\mathcal{F}')= \left(\End (\mathcal{F}) \setminus \{x,y\} \right) \cup \{x_1,x_2\}$, $|E\left(\mathcal{F}'[U_1 \cup U_2] \right) \Delta E\left(\mathcal{F}[U_1 \cup U_2] \right)| \leq 20$ and $|E(\mathcal{F}) \Delta E(\mathcal{F}')| = O(\log n)$. Furthermore, $\F'$ has no isolated vertices.
\end{claim}
\begin{proof}
First, apply Lemma \ref{lem:mainrotation} with $u:=x, U := V(G)$, and $V:= U_1 \cup U_2$ to give a $(V(G),O(\log n))$-rotation $\mathcal{F}_1$ of $\mathcal{F}$, where $x$ is replaced by a new endpoint $x_1 \in U_1 \cup U_2$, so that $\End (\F_1) = (\End (\F) \setminus \{x\}) \cup \{x_1\}$, and $\mathcal{F}[(U_1 \cup U_2) \setminus x_1] = \mathcal{F}_1[(U_1 \cup U_2) \setminus x_1]$. Furthermore, $\F_1$ has no isolated vertices. Without loss of generality, let $x_1 \in U_1$.
%, note that since $\mathcal{F}[(U_1 \cup U_2) \setminus x_1] = \mathcal{F}_1[(U_1 \cup U_2) \setminus x_1]$.
%$U_1 \cup U_2$ is an ($\mathcal{F}_1,C'-3)$-expander. 
 
Now, we apply again Lemma \ref{lem:mainrotation} to $\mathcal{F}_1$ with $u:=y,U := V(G) \setminus T,$ and $V := (U_1 \cup  U_2) \setminus T$, where $T$ is the set of vertices which are at most at distance $2$ from $x_1$ in $\F$, so that $|T| \leq 5$. 
Thus, we get a $O(\log n)$-rotation $\mathcal{F}_2$ of $\mathcal{F}_1$ where $y$ is replaced by a new endpoint $w \in (U_1 \cup U_2) \setminus T$, so that $\End (\F_2) = (\End (\F) \setminus \{x,y\}) \cup \{x_1,w\}$ and $\mathcal{F}_1[V \setminus \{w\}] = \mathcal{F}_2[V \setminus \{w\}]$. 
Furthermore, since the endpoint is not in $T$, it is different to $x_1$. Finally, $\F_2$ has no isolated vertices. 

Now, in the case that $w \in U_2$, the linear forest $\mathcal{F}' := \mathcal{F}_2$ is as desired with $x_2:= w$ and $x_1$ since $\mathcal{F}[(U_1 \cup U_2) \setminus (\{x_1,w\} \cup T)] = \mathcal{F}_2[(U_1 \cup U_2) \setminus (\{x_1,w\} \cup T)]$. 
Otherwise, if $w \in U_1$, then, by the previous condition that every vertex in $U_1$ has at least $100$ neighbours in $\inter_{\F}(U_2)$, there is an edge $wz$ for some $z \in \inter_{\F}(U_2) \cap \inter_{\mathcal{F}_2}(U_2)$. We can then use this edge to create a $1$-rotation $\F'$ of $\F_2$, replacing the old endpoint $w$ with a new endpoint $x_2 \in U_2 \setminus \{x_1\}$, that is, $\End (\F_2) = (\End (\F) \setminus \{x,y\}) \cup \{x_1,x_2\}$. 
The linear forest $\mathcal{F}'$ is then as desired since, in particular, $\mathcal{F}[(U_1 \cup U_2) \setminus (\{w,z,x_2\} \cup T)] = \mathcal{F}'[(U_1 \cup U_2) \setminus (\{w,z,x_2\} \cup T)]$ and again, by \Cref{obs:endpoint change}~\emph{\ref{obs: new endpoint is maybe iso}}, since $\F_2$ has no isolated vertices and $x_2$ is not an endpoint of $\F$, and so, also not one of $\F_2$ given that $\End (\F_2) = (\End (\F) \setminus \{x,y\}) \cup \{x_1,w\}$, $\F'$ has no isolated vertices. 

Finally, we have $|E\left(\mathcal{F}'[U_1 \cup U_2] \right) \Delta E\left(\mathcal{F}[U_1 \cup U_2] \right)| \leq 20$, since these linear forests differ in at most 8 vertices, either $T\cup\{x_1,w\}$ in the first case,  or $T\cup \{w,z,x_2\}$ in the second.
\renewcommand{\qedsymbol}{$\boxdot$}
\end{proof}
\renewcommand{\qedsymbol}{$\square$}
Let us make two remarks. First, the future rotations of $\F'$ that we will perform in the rest of the proof will all be $V(\mathcal{H})$-rotations, and, therefore, none of those linear forests will contain any isolated vertices; this is important since there will be applications of Lemma~\ref{obs:endpoint change}, %~(\emph{\ref{obs: new endpoint is maybe iso}},~\emph{\ref{obs: endpoints}})
\Cref{lem:rotationforclosing} and \Cref{lem:mainrotation} throughout the rest of the proof. Secondly, note that to obtain $\cal F'$ we rearranged some paths, hence, it is entirely possible that, in $\cal F'$, vertices in $U_1$ live in the same paths as vertices in $U_2$ (although this was not the case in $\F$). Thus, we cannot yet apply \Cref{lem:rotationforclosing} to the new endpoints $x_1,x_2$.

Now, let $\mathcal{H}'_3 \subseteq \mathcal{H}_3, \mathcal{H}'_4 \subseteq \mathcal{H}_4$, and $\mathcal{H}'_{\text{hop}} \subseteq \mathcal{H}_{\text{hop}},$  be the linear sub-forests of $\mathcal{H}_3,\mathcal{H}_4, \mathcal{H}_{\text{hop}}$ consisting of paths $P'_i$ such that $P_i \in \mathcal{F} \cap \mathcal{F}'$. 
    By the previous claim, we have that $|E(\mathcal{F}) \Delta E(\mathcal{F}')| = O(\log n)$. Since each path in $\mathcal{H}$ has size at most $\sqrt{n} = o(n/\log n)$, we must then have that $|V(\mathcal{H}'_j)| \geq 0.009n$ for each $j\in\{3,4,hop\}$. 
    By Lemma \ref{lem:forestexpander}, there exist an $(\mathcal{H}'_3,C')$-expander $U_3$, an $(\mathcal{H}'_4,C')$-expander $U_4$, and an $(\mathcal{H}'_{\text{hop}},C')$-expander $U_5$, each of size at least $0.009n-n/2C\geq 0.005n$. 
    As before, since we also can get that each $U_i$ is such that it $C'$-expands into $\inter_{\mathcal{F}'}(U_j)$ (for all $i,j\in\{3,4,hop\})$, we can have that every vertex in 
$U_{\text{hop}}$ has at least $C'\geq 1$ neighbours in $\inter_{\mathcal{F}'}(U_3)$. Also, note the following crucial observation.
\begin{claim}
$U_1$ and $U_2$ are $(\mathcal{F}',C'-40)$-expanders.    
\end{claim}
\noindent Indeed, this is easy to see since the claim before implies that $|E\left(\mathcal{F}'[U_1 \cup U_2] \right) \Delta E\left(\mathcal{F}[U_1 \cup U_2] \right)| \leq 20$. We now apply Lemma \ref{lem:mainrotation} again.
First, we apply the lemma to $\mathcal{F}'$ with $x:= x_1, U:= U_1,$ and $V:= U_{\text{hop}}$ (recall that $U_{\text{hop}}$ does not contain endpoints of $\F$, and so also of $\F'$), which gives a $(U_1 \cup U_{\text{hop}},O(\log n))$-rotation $\mathcal{F}'_1$ of $\mathcal{F}'$ in which the endpoint $x_1$ is replaced by an endpoint $h \in U_{\text{hop}}$. 
Moreover, we have that $\F'_1[U_{\text{hop}} \setminus h] = \F'[U_{\text{hop}} \setminus h]$ and $\F'_1$ is a $(U_1 \cup U_{\text{hop}})$-rotation, so that all broken edges in the successive rotations belong to $U_1 \cup U_{\text{hop}}$ and, thus, $U_2$ is still an $(\mathcal{F}'_1,C'-40)$-expander,  
$\inter_{\mathcal{F}'_1}(U_4) = \inter_{\mathcal{F}}(U_4)$ is of size at least $0.001n$, and $\F'_1[U_4] = \F'[U_4]$. 
Furthermore, the lemma implies that 
%$\F'_1$ contains no isolated vertices and that 
no vertex of $U_4$ is an endpoint of $\F'_1$. 
Therefore, we can apply again Lemma \ref{lem:mainrotation} to $\mathcal{F}'_1$, now with $x := x_2, U := U_2,$ and $V:= U_4$. This gives a $(U_2 \cup U_4,O(\log n))$-rotation $\mathcal{F}'_2$ with a new endpoint $x_4 \in U_4$ replacing $x_2$.
%and which contains no isolated vertices.

Now, note that for all paths $P'_i \in \mathcal{H}'_3$, we have that $P_i \in \mathcal{F}'_2$. Note also that since the previous rotations forming $\F'_1$ and $\F'_2$ are such that all the broken edges in the successive $1$-rotations apart from the last ones are not in $\F[U_{\text{hop}}]$ or $\F[U_4]$, we have that the path in $\mathcal{F}'_2$ containing $h$ is a sub-path of a path $P_i \in \F$ such that $P'_i$ is contained in $\HH'_{\text{hop}}$ and the path containing $x_4$ is a sub-path of a path $P_j \in \F$ such that $P'_j$ is contained in $\HH'_4$. 
Thus, $h,x_4$ are endpoints of different paths in $\mathcal{F}'_2$. Now, by assumption, $h \in U_{\text{hop}}$ has a neighbour in $\inter_{\mathcal{F}'}(U_3) = \inter_{\mathcal{F}'_2}(U_3)$.
%and $x_4 \in U_4$ has a neighbour in $\inter_{\mathcal{F}'}(U_6) = \inter_{\mathcal{F}'_2}(U_6)$. 
This implies that we can perform a $1$-rotation on $\mathcal{F}'_2$ to get $\mathcal{F}'_3$ and a vertex $x_3 \in U_3$ with $\End (\F'_3) = (\End(\F'_2) \setminus \{h\}) \cup \{x_3\} = (\End(\F) \setminus \{x,y\}) \cup \{x_3,x_4\}$. 
%Note that $x_3,x_4$ are not endpoints in $\F$ (and in $\F'$), and we have $\End(\F_2')=(\End(\F) \setminus \{x,y\}) \cup \{h,x_4\}$, so $x_3,x_4$ are not endpoints in $\F'_2$ either. Hence, in $\mathcal{F}'_3$, the vertices $x_3,x_4$ are also not isolated, implying that $\F_3'$ has no isolated vertex either, by \Cref{obs:endpoint change}~\emph{\ref{obs: new endpoint is maybe iso}}.

Now, since $h,x_4$ were endpoints of different paths in $\mathcal{F}'_2$, it is easy to check that there can be no path of $\mathcal{F}'_3$ which intersects both $U_3$ and $U_4$. Since $U_3$ and $U_4$ are still $(\mathcal{F}'_3,C'-2)$-expanders and have size at least $0.005n$, Lemma \ref{lem:rotationforclosing} implies the existence of a linear forest $\mathcal{F}'_4$ so that $\End (\mathcal{F}'_4)  = \End (\mathcal{F}) \setminus \{x,y\}$. Moreover, since every rotation performed was an $O(\log n)$-rotation, we have that $|E(\mathcal{F}) \Delta E(\mathcal{F}'_4)| = O(\log n)$, as desired. 
%, and furthermore no isolated vertices.  Moreover, since every rotation performed was an $O(\log n)$-rotation, we have that $|E(\mathcal{F}) \Delta E(\mathcal{F}'_4)| = O(\log n)$, as desired. 
\end{proof}

%%%%%%%%%%%%%%%%%%%%%%%%%%%%%%%5
%%%%%%%%%%%%%%%%%%%%%%%%%%%%%%%

%%%%%%%%%%%%%%%%%%%%%%%%%%%%

%%%%%%%%%%%%%%%%%%%%%%%%%%%%%%%%%%%%%%%%%%%%%%%%%%%%%%%%%%%%%%%%%%%%%%%%%%%
%%%%%%%%%%%%%%%%%%%%%%%%%%%%%%%%%%%%%%%%%%%%%%%%%%%%%%%%%%%%%%%%%%%%%%%%%%%

\section{Linear forests with few paths}\label{sec:shortpaths}

The goal of this section is to find a spanning linear forest in an expander with a relatively small number of paths. That is, we will prove the following result.

\begin{lem}\label{lem:smallforest}
Let $G$ be an $n$-vertex $C$-expander for $C>10^{10}$. Then, it contains a spanning linear forest with at most $n^{0.8}$ paths and no isolated vertices.
\end{lem}

\noindent We will find the required linear forest for Lemma~\ref{lem:smallforest} by taking a minimal such linear forest under a certain ordering, which is based on the lengths of the paths in the forest. In this section we will define the length of a path to be the number of vertices it has (which makes calculations a bit easier).  %Before we start, we will need the following definitions.

\begin{defn}
    Given a linear forest $\F$ and integers $a,b$, we define $S_\F(a,b)$ to be the set of vertices in the collection of paths in $\F$ of length at least $a$ and at most $b$. For each $v\in V(\F)$, we denote by $\seg_{\F}(v)$ the path in $\F$ which contains $v$. We omit $\F$ in the subscript when it is clear from the context.  
\end{defn}

\begin{defn}
    Let $<_{\lex}$ be the ordering on the family of all linear forests, where for any two linear forests $\F_1$ and $\F_2$ we have $\F_1<_{\lex} \F_2$ if
    \begin{itemize}
        \item $\F_1$ has fewer paths than $\F_2$, or
        \item $\F_1$ and $\F_2$ have the same number of paths and the vector of path lengths of $\F_1$ in decreasing order  is lexicographically smaller than that of $\F_2$.
    \end{itemize}
\end{defn}

%\begin{rem}
\noindent To illustrate this definition,  for example, if $\F_1$ consists of paths of length $8,5,3,2,2$ and $\F_2$ consists of paths of length $8,5,4,2,1$, then we have $F_1<_{\lex}F_2$ as they both have 5 paths but $(8,5,3,2,2)$ is smaller than $(8,5,4,2,1)$ in the lexicographic ordering.
%\end{rem}

% Before we state our main result from this section, we will prove a few simple lemmas used in its proof.
%\noindent Finally, we finish the section with two lemmas, the result of which will exclude linear forests with isolated vertices, which will simplify the technicalities of our proof later.

As discussed in Section~\ref{sec:rotations}, it will be convenient to use spanning linear forests with no isolated vertices. We will be able to apply the following lemma to show that a $<_{\lex}$-minimal spanning linear forest in an expander has no isolated vertices.

\begin{lem}\label{lem:induction_for_no_isolated_lemma}
Let $\F$ be a $<_{\lex}$-minimal spanning linear forest in an $n$-vertex $C$-expander $G$ for some $C>10^6$. Suppose that $\F$ contains an isolated vertex $v$, and $\F'$ is a $k$-rotation of $\F$ with old endpoint $v$. Then $\F'$ is $<_{\lex}$-minimal, the new endpoint $u$ of $\F'$ is isolated in $\F'$, and $\End(\F')=\End(\F)$.
\end{lem}
\begin{proof}
We prove this by induction on $k$.
For the initial case, ``$k=1$'', let $w$ be the pivot of the $1$-rotation, and let $P$ be the path of $\F$ containing $w$. Let $P_1$ and $P_2$ be the two (possibly empty) subpaths of $P-w$, labelled so that $|P_1|\leq |P_2|$. If $|P_2|\geq 2$, then replacing $\{v\}$ and $P$ in $\F$ with the paths $P_1wv$ and $P_2$ gives a linear forest $\F'$ with $\F'<_{\lex} \F$, which is a contradiction. Similarly, if $|P_1|=0$, then replacing $\{v\}$ and $P$ in $\F$ with the path $P_2wv$ gives  a linear forest $\F'$ with $\F'<_{\lex} \F$, which is a contradiction.
Thus we can suppose that $P_1, P_2$ are both single vertices, i.e.\ that $P=xwy$ for some $x,y$. Then, a $1$-rotation with pivot $w$ replaces the paths $P$ and $\{v\}$ with either the paths $vwx, y$ or the paths $vwy, x$. In both cases the new linear forest has the same path lengths as $\F$ did (and so stays $<_{\lex}$-minimal), the same set of endpoints as $\F$,  and has new endpoint either $x$ or $y$, which is now isolated.

For the induction step, consider a $k$-rotation $\F_k$ of $\F$ with old endpoint $v$. By the definition of a $k$-rotation, we get a $(k-1)$-rotation $\F_{k-1}$ of $\F$ with old endpoint $v$, new endpoint $u$, such that $\F_k$ is a $1$-rotation of $\F_{k-1}$ with old endpoint $u$. By induction we get that $\F_{k-1}$ is $<_{\lex}$-minimal, has $u$ isolated, and has $\End(\F_{k-1})=\End(\F)$. By the ``$k=1$'' case, we have that $F_k$ is $<_{\lex}$-minimal, its new endpoint is isolated, and $\End(\F_{k})=\End(\F_{k-1})=\End(\F)$.
\end{proof}
\noindent We now sketch out the proof of Lemma~\ref{lem:smallforest}. Given a $<_{\lex}$-minimal spanning linear forest $\F$ in a $C$-expander $G$, with $C$ large, we will first observe that there cannot be an edge in $G$ from the end of a path in $\F$ to a path in $\F$ which is more than five times as long (see Lemma~\ref{lem:small increase paths}). We will then argue, for any $i\geq 0$, that this implies that if we perform a sequence of $i$ 1-rotations starting with an endpoint of a smallest path in a $<_{\lex}$-minimal spanning linear forest $\F$, then the new endvertex must be in a path at most $6^i$ times as long as a smallest path (see Lemma~\ref{Lemma_segment_lengths_after_rotation}). However, Lemma~\ref{Lem_endpoints_in_expander} implies that within at most $k:=2\log_Cn$ rotations we reach at least $n/10^5$ new endpoints, which then collectively have edges to all but at most $n/2C$ vertices in the graph by the definition of a $C$-expander. After using Lemma~\ref{lem:induction_for_no_isolated_lemma} and assuming, for sake of contradiction that $\F$ has more than $n^{0.8}$ paths, we can deduce that $\F$ has no isolated vertices and also note that $\F$ must have at most $100n/C$ vertices contained in paths of length less than $100$, so that the $<_{\lex}$-minimality of $\F$ then implies by Lemma~\ref{lem:generalpathforestrot} that at least $0.8n$ of the vertices of $G$ are in paths in $\F$ with more than $\sqrt{n}$ vertices. In combination, this all implies that the shortest path in $\F$ must have at least $6^{-(k+1)}\sqrt{n}$ vertices, so that $\F$ contains at most $\sqrt{n}\cdot 6^{(k+1)}\leq n^{0.8}$ paths, a contradiction, as required.

First, then, we show endvertices of paths cannot have an edge to much longer paths in a $<_{\lex}$-minimal spanning linear forest, in the following stronger form.

\begin{lem}\label{lem:small increase paths}
    Let $t\geq t'>0$ be integers. Let $\F$ be a minimal spanning linear forest in $G$ with respect to $<_{\lex}$, with paths of decreasing lengths $\ell_1,\ldots,\ell_t$. Let $\F'$ be another forest with paths of lengths $s_1,\ldots,s_t$ in decreasing order, where the first $t'-1$ lengths are the same as in $\F$, i.e.\ $\ell_i=s_i$ for each $i<t'$. Let $x$ be an endpoint in $\F'$ in a path of length at most $s_{t'}$. Then, $N_G(x)\subseteq S_{\F'}(0,5s_{t'})$.
\end{lem}

\begin{proof}
    Suppose for contradiction there is a path $P_i $ (of length $s_i$) in $\F'$ with $s_i\leq s_{t'}$, whose endpoint $x$ has a neighbour $y$ in another path $P_j$ of length at least $5s_{t'}\geq 5|P_i|\geq 5$. Then, we can obtain a new forest $\F''$ from $\F'$ such that $\F''<_{\lex} \F$, as follows. We replace the paths $P_i$ and $P_j$ in $\F'$ by two new paths. The first new path is the concatenation of $P_i$, the edge $xy$ and the shorter subpath of $P_j$ which starts at $y$ and ends in an endpoint of $P_j$. The other path is the remaining part of $P_j$. Note that both of those paths are shorter than $P_j$ as the first path is of length at most $|P_i|+1+|P_j|/2\leq |P_j|/5+1+|P_j|/2<|P_j|$, and the second one is a strict subpath of $P_j$. This contradicts the $<_{\lex}$-minimality of $\F$.
\end{proof}

\noindent Next, we iterate Lemma~\ref{lem:small increase paths}, showing that repeatedly applying 1-rotations cannot create a new endpoint with a neighbour in a considerably longer path than the path containing the old endpoint, as follows. 

\begin{lem}\label{Lemma_segment_lengths_after_rotation}
Let $\F$ be a $<_{\lex}$-minimal spanning linear forest of $G$. Let $v$ be an endpoint of a path in $\F$. Then $N(E^k_G(v,\F))\subseteq S_{\F}(0, 6^{k+1}|\seg_{\F}(v)|)$ for every positive integer $k$.
\end{lem}

\begin{proof}
    Let $\F_k$ be a $k$-rotation of $\F$ and let $u_k$ be its new endpoint and $v=u_0$ its old endpoint. Consider the linear forests $\F=\F_0,\F_1,\F_2,\ldots,\F_k$, where $\F_i$ is a $1$-rotation of $\F_{i-1}$, and where the new endpoint of $\F_i$ is $u_i$, the old endpoint is $u_{i-1}$, and the pivot is $z_i$. We show that the following two conditions hold for every $i\geq 0$:
    \begin{enumerate}[label = (\arabic{enumi})]
        \item\label{ind:neighbourhood} $N(u_i)\subseteq S_\F(0,6^{i+1}|\seg_\F(v)|)$ 
        \item\label{ind:large path lengths} The paths of length $>6^{i+1}|\seg_\F(v)|$ are the same in $\F$ and $\F_{i+1}$.
    \end{enumerate}
    Hence, the particular case $i=k$ proves the statement of the lemma. For the base case $i=0$, let $\ell_1,\ldots,\ell_t$ be the lengths of the paths in $\F$ in decreasing order. Let $t'$ be such that $v$ is in a path of length $\ell_{t'}$.
    Note that by \Cref{lem:small increase paths} applied to $\F=\F'$ with $t'$, we know that $v$ only has neighbours in $S_{\F}(0,5|\seg_\F(v)|)\subseteq S_{\F}(0,6|\seg_\F(v)|)$, as required for part~\ref{ind:neighbourhood}. In the resulting linear forest $\F_1$ only two paths have changed compared to $\F$ -- one of length $|\seg_{\F}(v)|$ and one of length at most $5|\seg_{\F}(v)|$, and so we obtained two new paths of length at most $6|\seg_{\F}(v)|$ each. Hence, all the paths in $\F_1$ of length larger than $6|\seg_{\F}(v)|$ are the same as in $\F$, proving part~\ref{ind:large path lengths}.

    Suppose now the claim holds for all $j<i$, and let us prove it for $i$. Note that the first condition implies that all the pivots $z_j$ for $j\leq i$ are in $S_{\F}(0,6^j|\seg_\F(v)|)$, as each $z_j$ lives in $N(u_{j-1})$. In particular, $z_i\in S_{\F}(0,6^i|\seg_\F(v)|)$, meaning that $u_i\in S_{\F}(0,6^i|\seg_\F(v)|)$, because either $z_i=u_i$ or the broken edge $z_iu_i$ is in $\F$, so $u_i$ and $z_i$ are on the same path in $\F$. 
    Additionally, note that by the induction hypothesis (part~\ref{ind:large path lengths}) all the paths of length larger than $6^{i}|\seg_\F(v)|$ are the same in $\F$ and $\F_i$. Hence, since $u_i\in S_{\F}(0,6^i|\seg_\F(v)|)$, then $u_i\in S_{\F_i}(0,6^i|\seg_\F(v)|)$. Therefore we can apply \Cref{lem:small increase paths} to $\F,\F_i$ and $u_i$ to get that $N(u_i)\subseteq S_{\F_i}(0,5|\seg_{\F_i}(u_i)|)$, which by the previous discussion gives that $N(u_i)\subseteq S_{\F_i}(0,5\cdot 6^{i}|\seg_{\F}(v)|)$.
    Now, we show that~\ref{ind:large path lengths} holds. Note that, in order to get $\F_{i+1}$, in $\F_i$ the path which contains $u_i$ and a path which contains a vertex in $N(u_i)$ are replaced with two paths on the same set of vertices. Thus the two newly obtained paths are of length at most $6^i|\seg_\F(v)|+5\cdot6^i|\seg_\F(v)|\leq 6^{i+1}|\seg_\F(v)|$. As the paths of length $>6^i|\seg_\F(v)|$ are the same in $\F$ and $\F_i$, this means that the paths of length $>6^{i+1}|\seg_\F(v)|$ are the same in $\F$ and $\F_{i+1}$, completing part~\ref{ind:large path lengths}.

    For part~\ref{ind:neighbourhood}, recall that $N(u_i)\subseteq S_{\F_i}(0,5 \cdot 6^{i}|\seg_{\F}(v)|)$.
    %we can apply \Cref{lem:small increase paths} to $u_i$, $\F$ and $\F_i$ to get that $N(u_i)\subseteq S_{\F_i}(0,5|seg_{\F_i}(v)|)\subseteq S_{\F_i}(0,5\cdot6^i|\seg_\F(v)|)$.
    Thus, $N(u_i)\subseteq S_{\F}(0,6^{i+1}|\seg_\F(v)|)$. Indeed, otherwise a vertex $x\in N(u_i)$ is in a path of length $>6^{i+1}|\seg_\F(v)| \geq 6^{i}|\seg_\F(v)|$ in $\F$, but in $\F_i$ it is in a path of length at most $6^{i+1}|\seg_\F(v)|$. By the induction hypothesis for part~\ref{ind:large path lengths} for $i-1$, the paths of length $> 6^{i}|\seg_\F(v)|$ are the same in $\F$ and $\F_i$ and thus, this is a contradiction. This completes the proof.
\end{proof}

\noindent Next, we show that in a $<_{\lex}$-minimal spanning linear forest with quite a lot of paths, only a small proportion of the vertices can lie in very long paths.

\begin{lem}\label{lem:not many long paths}
Let $C > 10^{10}$. Let $\F$ be a $<_{\lex}$-minimal spanning linear forest in an $n$-vertex $C$-expander $G$. Suppose that the number of paths in $\F$ is at least $n^{\varepsilon}$, for some $\varepsilon>0$.
Then, $|S(6n^{1-\varepsilon+2\log_C6},n)|\leq n/2C$.
\end{lem}
\begin{proof}
Since we have $n^\varepsilon$ paths, there is one of length at most $n^{1-\varepsilon}$. Set $k:=2\log_{C}n$, let $v$ be an endpoint of a shortest path, which is of length $m \leq n^{1-\varepsilon}$, and suppose for contradiction that $| S(6mn^{2\log_C6},n)|> n/2C$.
By Lemma~\ref{Lem_endpoints_in_expander}, we have $|E^k_G(v,\F)|\geq n/2C$. 
We now have two cases. If $E^k_G(v,\F)\cap S(6mn^{2\log_C6},n)\neq \emptyset$, then let $i$ be the smallest such that there exists $x\in E^i_G(v,\F)\cap S(6mn^{2\log_C6},n)$. Let $y$ be the pivot for $x$ in the corresponding $i$-rotation and recall that $x,y$ are in the same path in $\F$, hence $y\in S(6mn^{2\log_C6},n)$ as well. Note also that $y\in N(E^{i-1}_G(v,\F))$, so that $S(6mn^{2\log_C6},n) \cap N(E^{i-1}_G(v,\F)) \neq \emptyset$.
On the other hand, if $E^k_G(v,\F)\cap S(6mn^{2\log_C6},n) = \emptyset$, then by the definition of a $C$-expander, we have that $N(E^k_G(v,\F))\cap S(6mn^{2\log_C6},n) \neq \emptyset$. 
Both of these cases contradict Lemma~\ref{Lemma_segment_lengths_after_rotation}, which says that for $i\leq k$ we have
\[N(E^i_G(v,\F))\subseteq S(0, 6^{k+1}|\seg_{\F}(v)|)=S(0, 6^{1+2\log_Cn}m)=S(0, 6^{1+2\log_6n/\log_6C}m)=S(0, 6mn^{2\log_C6}).
\]
Hence, $|S(6n^{1-\varepsilon+2\log_C6},n)|\leq | S(6mn^{2\log_C6},n)| \leq n/2C$.
\end{proof}

\noindent Finally, then, in this section, we can put this all together to prove Lemma~\ref{lem:smallforest}.

\begin{proof}[ of Lemma~\ref{lem:smallforest}]
Let $\F$ be a $<_{\lex}$-minimal spanning forest in $G$. First we will show that it must contain at most $n/C$ paths. Indeed, suppose otherwise and let $S$ be a set which contains exactly one endpoint from each path in $\F$. Split $S$ into disjoint sets $S=S_1\cup S_2$, both of size $|S_1|,|S_2|\geq n/2C$. Since $G$ is a $C$-expander, there is an edge between $S_1$ and $S_2$. Adding this edge to $\F$ creates a linear forest with one fewer path than $\F$, contradicting the $<_{\lex}$-minimality of $\F$.

We now show that $\F$ has no isolated vertices. Again, suppose otherwise and let 
$v$ be an isolated vertex in $\F$. By Lemma~\ref{lem:induction_for_no_isolated_lemma}, for all $k$, we have that  $E^k(v,\F)\subseteq \End(\F)$. Lemma~\ref{Lem_endpoints_in_expander} then tells us that $|\End(\F)|\geq |E^{2 \log_C n}(v,\F)| \geq n/10^5 > 2n/C$. This contradicts the previous assertion that $\F$ has at most $n/C$ paths.

Now that we know that $\F$ has 
at most $n/C$ paths and
no isolated vertices, it must be that at most $100n/C$ vertices in $\F$ are contained in paths of length less than $100$.
To conclude, assume for contradiction that this forest has more than $n^{0.8}$ paths.
Now, apply Lemma~\ref{lem:not many long paths} with $\varepsilon=0.8$ to conclude that $\F$ satisfies $|S_\F(6n^{1-\varepsilon+2\log_C6 },n)|\leq n/2C$. Since $1-\varepsilon+2\log_C6 <1/2$, the number of vertices in paths longer than $\sqrt{n}$ is less than $n/2C<n/2$.
Hence, at least $n/2-100n/C\geq 0.1n$ of the vertices are in $S_{\F}(100,\sqrt{n})$. Now, apply Lemma \ref{lem:generalpathforestrot} to $\F$ to obtain a new linear forest with fewer paths, contradicting the $<_{\lex}$-minimality of $\F$. 
\end{proof}
%%%%%%%%%%%%%%%%%%%%%%%%%%%%%%%%%%%%%%%%%%%%55
%%%%%%%%%%%%%%%%%%%%%%%%%%%%%%%%%%%%%%%%%%%%%%

\section{Linking structures and decomposing the expander graph}\label{sec:linking}
In this section, we will prove the main result we need which will find an appropriate linking structure in an expander. For convenience for our application, we will additionally find relatively long paths connecting the two special sets in the linking structure, as in the following result.
\begin{lem}\label{lem:partition:new}
Let $C > 10^{15}$, let $n$ be sufficiently large and let $G$ be a $C$-expander. Then, there is a partition of $V(G)$ into three sets $X,Y,Z$, and disjoint sets $A,B\subset X$, with the following properties.
\begin{itemize}
    \item $G[X]$ is an $(A,B)$-linking structure and $|A|=|B|=n^{0.9}$.
    \item $G[Y \cup A \cup B]$ has a spanning linear forest with $|A| = |B|$ paths of size $n^{0.1}/5$ whose endpoints are in $A \cup B$.
    \item $G[Z]$ and $G[Z \cup Y \cup A \cup B]$ are both $C/100$-expanders.
\end{itemize}
\end{lem}
\noindent For \cite[Proposition 4.4]{hyde2023spanning}, Hyde, Morrison, M\"uyesser and Pavez-Sign\'e showed that there exist linking structures with convenient properties, which allow for them to be constructed in pseudorandom graphs. Our methods would be content with a linking structure with much smaller sets $A,B$, than in the efficiently constructed linking structure in~\cite{hyde2023spanning}, but we will use the same construction, briefly outlining the main arguments for the sake of completeness.
We will start in Section~\ref{sec:embeddinginexpanders} by recalling some `extendability' methods, before proving in Section~\ref{subsec:linking} that there is a linking structure that can be constructed using these methods, and using this to prove Lemma~\ref{lem:partition:new}.%however with the following section containing some results allowing us to find this linking structure in our expander graph.

%%%%%%%%%%%%%%%%%%%%%%%%%%%%%%%%%%%%%%%%%%%%%%%
%%%%%%%%%%%%%%%%%%%%%%%%%%%%%%%%%%%%%%%%%%%%

\subsection{Embedding in expanders}\label{sec:embeddinginexpanders}
\noindent Here we briefly discuss a very useful `extendability' method for embedding sparse graphs in expander graphs. The method combines a technique Friedman and Pippenger~\cite{friedman1987expanding} used to inductively embed trees leaf-by-leaf into expanding graphs with a `roll-back' idea of Johannsen (used to prove Lemma~\ref{lem:adding a path to extedable graph} below). For more on this technique, see~\cite{draganic2022rolling}. The key definition defining a type of `good' embedding is the following.
\begin{defn}\label{def:extendability}Let $D,m$ be positive integers with $D\ge 3$.
Let $G$ be a graph and let $H\subset G$ be a subgraph with $\Delta(H) \leq D$.
Then $H$ is \emph{$(D,m)$-extendable (in $G$)}
if for every $S\subset V(G)$ with $1\le |S|\le 2m$ we have
    \begin{equation}%\label{def:extendability}
        |\Gamma_G(S)\setminus V(H)|\ge (D-1)|S|-\sum_{u\in S\cap V(H)}(d_H(u)-1).
    \end{equation}
\end{defn}
\noindent The main result to state here is the following, which allows us to add paths to existing $(D,m)$-extendable subgraphs of an expander such that the resulting subgraph stays $(D,m)$-extendable. Hence, we will be able to embed in expanders any type of sparse graphs which can be recursively built by adding path by path. This lemma is, for example, a weaker version of \cite[Corollary 3.12]{montgomery2019spanning}.
\begin{lem}\label{lem:adding a path to extedable graph} Let $C > 10D$ and let $G$ be an $n$-vertex $C$-expander with a $(D,n/2C)$-extendable subgraph $H$ with $|H| \leq n - 5nD/C - \ell$, where $\ell \geq \log n$. Then, the following hold for all vertices $y \in V(H)$ with $\text{deg}_H(y) \leq D/2$.
\begin{itemize}
    \item There exists a path $P$ in $G$ with endpoint $y$ of length $\ell$ such that all its vertices except for $y$ lie outside of $H$ and $H \cup P$ is $(D,n/2C)$-extendable.
    \item For every $x \in V(H)\setminus \{y\}$ with $\textrm{deg}_H(x) \leq D/2$, there exists an $xy$-path $P$ in $G$ of length $\ell$ such that all its internal vertices lie outside of $H$ and $H \cup P$ is $(D,n/2C)$-extendable.
\end{itemize}
\end{lem}

%%%%%%%%%%%%%%%%%%%%%%%%%%%%%%%%%%%%%%%%%%%%%%%%%%%%%%%%%%%%%

\subsection{Linking structures in expanders}\label{subsec:linking}
 We first need the following definition, as introduced in \cite{hyde2023spanning}.
\begin{defn}\label{Def_constructible}
    Let $G$ be a graph and $A\subseteq V(G)$. We say that $G$ is \emph{$A$-path-constructible} if there exists a sequence of edge-disjoint paths $P_1,\ldots, P_t$ in $G$ with the following properties.
    \begin{enumerate}
        \item $E(G)=\bigcup_{j\in [t]} E(P_j)$.
        \item  For each $i\in [t]$, the internal vertices of $P_i$ are disjoint from $A\cup \bigcup_{j\in [i-1]}V(P_j)$.
        \item For each $i\in [t]$, at least one of the endpoints of $P_i$ belongs to $A\cup \bigcup_{j\in [i-1]}V(P_j)$.
    \end{enumerate}    
We will say that $G$ is $A$-path-constructible \emph{with paths of length between $\ell_1$ and $\ell_2$} if all the paths $P_t$ have lengths between $\ell_1$ and $\ell_2$.
\end{defn}
\noindent The crucial property of path-constructible graphs is simple to observe given Section \ref{sec:embeddinginexpanders}: if they have low maximum degree and the paths $P_i$ are sufficiently long (that is, of size $\Omega(\log n)$), then Lemma \ref{lem:adding a path to extedable graph} implies that they can be embedded in expanders (and, in particular, in pseudorandom graphs, as done in \cite{hyde2023spanning}). As mentioned before, we now state the main lemma for finding our linking structure, which is a weaker version of  \cite[Proposition 4.4]{hyde2023spanning}. 
\begin{lem}\label{lem:linkingstructure}
For all sufficiently large $N$, there exists a graph $H$ with at most $N^{1.1}$ vertices and disjoint sets $A,B \subseteq V(H)$ such that $H$ is an $(A,B)$-linking structure and the following hold.
\begin{enumerate}
    \item $|A|= |B| = N$ and $A \cup B$ is an independent set in $H$.
    \item $\Delta(H) \leq 4$.
    \item $H$ is $(A \cup B)$-path-constructible with paths of length between $10 \log N$ and $40 \log N$.
\end{enumerate}
\end{lem}
\noindent In order to obtain their stronger version of Lemma \ref{lem:linkingstructure}, the authors in \cite{hyde2023spanning} rely on a result concerning sorting networks with optimal depth. Earlier uses of sorting networks in extremal combinatorics include in the work of K\"uhn, Lapinskas, Osthus and Patel~\cite{kuhn2014proof} on Hamilton cycles in highly connected tournaments and in the work of M\"{u}yesser and Pokrovskiy \cite{muyesser2022random} giving, among other results, a combinatorial proof of the Hall-Paige conjecture. We will define directly the graph theoretic counterpart of a sorting network, for convenience calling this itself a \emph{sorting network}. For a more detailed discussion concerning sorting networks and their graph theoretic counterpart see, for example, \cite[Section 4]{hyde2023spanning}.
%%, which is of our interest.  For simplicity, we will make some abuse of nomenclature and call this graph theoretic counterpart the \emph{sorting network}. 
\begin{defn}
A graph $G$ is an \emph{$(N,\ell)$-sorting network} if there exist $A,B \subseteq V(G)$ for which $G$ is an $(A,B)$-linking structure and disjoint sets $V_0=A,V_1, \ldots,V_{\ell-1},V_\ell= B$ with the following properties.
\begin{enumerate}
\item $V(G)=V_0 \cup \ldots \cup V_\ell$.
\item For every $0\leq i\leq \ell$, $V_i$ is an independent set in $G$ with $|V_i| = |A| = |B| = N$.
\item For every $0\leq i\leq \ell$, the bipartite graph $G[V_i,V_{i+1}]$ is the disjoint union of $K_{2,2}$'s and edges.
\end{enumerate}
\end{defn}
\noindent  
Sorting networks which are efficient enough for our purposes are simple to construct, and results as early as 1959 (\cite{shell1959high}) imply that there exist $(N,O(\log^2 N))$-sorting networks. Moreover, in 1983, Ajtai, Koml\'os and Szemer\'edi~\cite{ajtai19830} showed that, for each $N$, there exist $(N,\ell)$-sorting networks with $\ell=O(\log N)$, where $\ell$ can easily seen to be optimal up to a constant multiple.
The sorting network in the above definition is not necessarily the linking structure we require for Lemma~\ref{lem:linkingstructure} --  although it has a low maximum degree, it might not be $(A \cup B)$-path-constructible. As observed in \cite{hyde2023spanning}, it is possible to `transform' a sorting network into an $(A \cup B)$-path-constructible linking structure. Briefly, to do this, start by taking the graph $G$ with the properties above with $|A| = |B| = N$ and $\ell = O(\log N)$. Then, for each $0\leq i<\ell$, take the bipartite graph $G[V_i,V_{i+1}]$ and, using part 3 of the definition above, substitute each single edge by a path of an appropriately chosen length $m$ and each $K_{2,2}$ with vertices $a_1,a_2 \in V_i$ and $b_1,b_2 \in V_{i+1}$ by the $(2m)$-vertex graphs in Figure~\ref{fig:gadget}.%ollowing graph with in total $2m$ vertices, which is illustrated and explained below. 
\begin{figure}[!ht]
\begin{center}
\begin{tikzpicture}[xscale = 0.9]
\draw [line width=1pt, color=red] (-4,6)-- (-3.5,4);
\draw [line width=1pt, color=blue] (-4,4)-- (-2.5,6);
\draw [line width=1pt, color=red] (-1,6)-- (-2.5,4);
\draw [line width=1pt, color=blue] (-3.5,6)-- (-2,4);
\draw [line width=1pt, color=blue] (0.5,6)-- (-1,4);
\draw [line width=1pt, color=red] (-2,6)-- (-0.5,4);
\draw [line width=1pt, color=blue] (-0.5,6)-- (1,4);
\draw [line width=1pt, color=red] (1.5,5)-- (0.5,4);
\draw [line width=1pt, color=red] (1.5,5.5)-- (1,6);
\draw [line width=1pt, color=blue] (3.5,5)-- (4.5,4);
\draw [line width=1pt, color=blue] (3.5,5.5)-- (4,6);
\draw [line width=1pt, color=red] (4,4)-- (5.5,6);
\draw [line width=1pt, color=red] (6,4)-- (4.5,6);
\draw [line width=1pt, color=red] (7,4)-- (8.5,6);
\draw [line width=1pt, color=red] (9,4)-- (7.5,6);
\draw [line width=1pt, color=blue] (5.5,4)-- (7,6);
\draw [line width=1pt, color=blue] (7.5,4)-- (6,6);
\draw [line width=1pt, color=blue] (8.5,4)-- (9,6);

\draw [fill=black] (2,5) circle (1pt);
\draw [fill=black] (2.5,5) circle (1pt);
\draw [fill=black] (3,5) circle (1pt);

\begin{scriptsize}

\draw [line width=1pt,dash pattern=on 2pt off 2pt, color=blue] (-6,4)-- (-4,4);
\draw [line width=1pt] (-4,4)-- (-3.5,4);
\draw [line width=1pt,dash pattern=on 2pt off 2pt, color=red] (-3.5,4)-- (-2.5,4);
\draw [line width=1pt] (-2.5,4)-- (-2,4);
\draw [line width=1pt,dash pattern=on 2pt off 2pt, color=blue] (-2,4)-- (-1,4);
\draw [line width=1pt] (-1,4)-- (-0.5,4);
\draw [line width=1pt,dash pattern=on 2pt off 2pt, color=red] (-0.5,4)-- (0.5,4);
\draw [line width=1pt] (0.5,4)-- (1,4);
\draw [line width=1pt,dash pattern=on 2pt off 2pt, color=blue] (1,4)-- (1.5,4);

\draw [line width=1pt,dash pattern=on 2pt off 2pt, color=red] (11,4)-- (9,4);
\draw [line width=1pt] (9,4)-- (8.5,4);
\draw [line width=1pt,dash pattern=on 2pt off 2pt, color=blue] (8.5,4)-- (7.5,4);
\draw [line width=1pt] (7.5,4)-- (7,4);
\draw [line width=1pt,dash pattern=on 2pt off 2pt, color=red] (7,4)-- (6,4);
\draw [line width=1pt] (6,4)-- (5.5,4);
\draw [line width=1pt,dash pattern=on 2pt off 2pt, color=blue] (5.5,4)-- (4.5,4);
\draw [line width=1pt] (4.5,4)-- (4,4);\draw [line width=1pt,dash pattern=on 2pt off 2pt, color=red] (4,4)-- (3.5,4);

\draw [line width=1pt,dash pattern=on 2pt off 2pt, color=red] (-6,6)-- (-4,6);
\draw [line width=1pt] (-4,6)-- (-3.5,6);
\draw [line width=1pt,dash pattern=on 2pt off 2pt, color=blue] (-3.5,6)-- (-2.5,6);
\draw [line width=1pt] (-2.5,6)-- (-2,6);
\draw [line width=1pt,dash pattern=on 2pt off 2pt, color=red] (-2,6)-- (-1,6);
\draw [line width=1pt] (-1,6)-- (-0.5,6);
\draw [line width=1pt,dash pattern=on 2pt off 2pt, color=blue] (-0.5,6)-- (0.5,6);
\draw [line width=1pt] (0.5,6)-- (1,6);
\draw [line width=1pt,dash pattern=on 2pt off 2pt, color=red] (1,6)-- (1.5,6);

\draw [line width=1pt,dash pattern=on 2pt off 2pt, color=blue] (11,6)-- (9,6);
\draw [line width=1pt] (9,6)-- (8.5,6);
\draw [line width=1pt,dash pattern=on 2pt off 2pt, color=red] (8.5,6)-- (7.5,6);
\draw [line width=1pt] (7.5,6)-- (7,6);
\draw [line width=1pt,dash pattern=on 2pt off 2pt, color=blue] (7,6)-- (6,6);
\draw [line width=1pt] (6,6)-- (5.5,6);
\draw [line width=1pt,dash pattern=on 2pt off 2pt, color=red] (5.5,6)-- (4.5,6);
\draw [line width=1pt] (4.5,6)-- (4,6);

\draw [line width=1pt,dash pattern=on 2pt off 2pt, color=blue] (3.5,6)-- (4,6);
\draw [fill=black] (11,6) circle (2.5pt);
\draw [fill=black] (9,6) circle (2.5pt);
\draw [fill=black] (8.5,6) circle (2.5pt);
\draw [fill=black] (7.5,6) circle (2.5pt);
\draw [fill=black] (7,6) circle (2.5pt);
\draw [fill=black] (6,6) circle (2.5pt);
\draw [fill=black] (5.5,6) circle (2.5pt);
\draw [fill=black] (4.5,6) circle (2.5pt);
\draw [fill=black] (4,6) circle (2.5pt);
\draw [fill=black] (-6,6) circle (2.5pt);
\draw [fill=black] (-4,6) circle (2.5pt);
\draw [fill=black] (-3.5,6) circle (2.5pt);
\draw [fill=black] (-2.5,6) circle (2.5pt);
\draw [fill=black] (-2,6) circle (2.5pt);
\draw [fill=black] (-1,6) circle (2.5pt);
\draw [fill=black] (-0.5,6) circle (2.5pt);
\draw [fill=black] (0.5,6) circle (2.5pt);
\draw [fill=black] (1,6) circle (2.5pt);
\draw [fill=black] (0.5,4) circle (2.5pt);
\draw [fill=black] (1,4) circle (2.5pt);
\draw [fill=black] (11,4) circle (2.5pt);
\draw [fill=black] (9,4) circle (2.5pt);
\draw [fill=black] (8.5,4) circle (2.5pt);
\draw [fill=black] (7.5,4) circle (2.5pt);
\draw [fill=black] (7,4) circle (2.5pt);
\draw [fill=black] (6,4) circle (2.5pt);
\draw [fill=black] (5.5,4) circle (2.5pt);
\draw [fill=black] (4.5,4) circle (2.5pt);
\draw [fill=black] (4,4) circle (2.5pt);
\draw [fill=black] (-6,4) circle (2.5pt);
\draw [fill=black] (-4,4) circle (2.5pt);
\draw [fill=black] (-3.5,4) circle (2.5pt);
\draw [fill=black] (-2.5,4) circle (2.5pt);
\draw [fill=black] (-2,4) circle (2.5pt);
\draw [fill=black] (-1,4) circle (2.5pt);
\draw [fill=black] (-0.5,4) circle (2.5pt);

\node[color=black, scale = 1.5]  at (-6.4,6) {$a_1$};
\node[color=black, scale = 1.5]  at (-6.4,4) {$a_2$};
\node[color=black, scale = 1.5]  at (11.4,6) {$b_1$};
\node[color=black, scale = 1.5]  at (11.4,4) {$b_2$};

\end{scriptsize}
\end{tikzpicture}
\caption{Each $K_{2,2}$ in the sorting network with parts $\{a_1,a_2\}$ and $\{b_1,b_2\}$ is substituted with a graph of the type above. Each full line represents an edge and each dotted line represents a path of an appropriately chosen length. The pattern in the figure is such that we can always construct such a graph with any number of marked vertices (i.e., those drawn with a circle in the figure) which is sufficiently large and divisible by $4$. Note that the union of the full lines is simply an even  cycle.} 
\label{fig:gadget}
\end{center}
\end{figure}
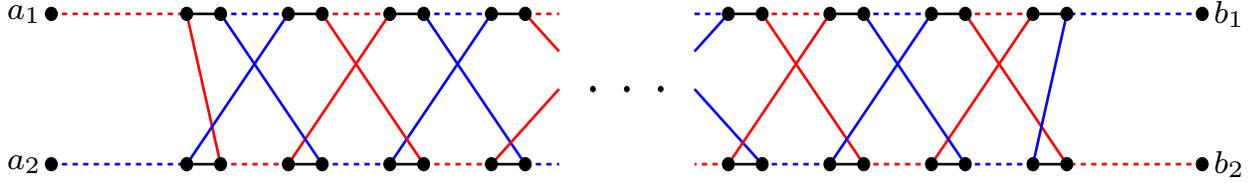

\noindent Notice that the graph in Figure~\ref{fig:gadget} is an $(\{a_1,a_2\},\{b_1,b_2\})$-linking structure. Indeed, the two horizontal paths  are vertex disjoint paths linking $a_1$ to $b_1$ and $a_2$ to $b_2$, and they span all the vertices, and the blue and red paths are vertex disjoint paths linking $a_1$ to $b_2$ and $a_2$ to $b_1$, and they span all the vertices. Since each of these graphs is a linking structure, substituting each $K_{2,2}$ in each $G[V_i,V_{i+1}]$ with one of these graphs and each edge in each $G[V_i,V_{i+1}]$ with a path, maintains that $G$ is an $(A,B)$-linking structure. 

The crucial property of the graph  in Figure~\ref{fig:gadget} is that it is also $\{a_1,a_2\}$-path-constructible (as well as $\{a_1,a_2,b_1,b_2\}$-path-constructible). Indeed, it can be constructed by first taking $a_1$ and adding its incident path represented by a dotted line, then, constructing the even cycle in the middle formed by the full lines (which represent edges), and, finally, adding all of the remaining paths represented by dotted lines. Furthermore, substituting each $K_{2,2}$ in each $G[V_i,V_{i+1}]$ with one of the graphs in Figure~\ref{fig:gadget} and each edge in each $G[V_i,V_{i+1}]$ with a path, transforms $G$ into an $A \cup B$-path-constructible graph. 

We can now complete our sketch proof of Lemma~\ref{lem:linkingstructure} by taking the illustrated graph to be such that it has between $10 \log N$ and $40 \log N$ marked vertices (which as is explained in the caption of the figure is possible) and the dotted lines to represent equally-sized paths with length also between $10 \log N$ and $40 \log N$. We then have that the illustrated graph has $2m = \Theta(\log^2 N)$ vertices and it is $\{a_1,a_2\}$-path-constructible (and $\{a_1,a_2,b_1,b_2\}$-path-constructible) with paths of length between $10 \log N$ and $40 \log N$. Then, substituting each $K_{2,2}$ in each $G[V_i,V_{i+1}]$ with one of these illustrated graphs and each edge in each $G[V_i,V_{i+1}]$ with a path of length $m$ makes it so that the resulting graph is an $(A,B)$-linking structure with $O(N \log^3 N) \leq N^{1.1}$ vertices and is $(A \cup B)$-path-constructible with paths of length between $10 \log N$ and $40 \log N$. Thus, it satisfies Lemma \ref{lem:linkingstructure}.

Using Lemma~\ref{lem:linkingstructure}, Lemma~\ref{lem:adding a path to extedable graph} and Definition~\ref{def:extendability}, we can now prove Lemma \ref{lem:partition:new}.
\begin{proof}[ of Lemma~\ref{lem:partition:new}]
First, we will find $A,B \subseteq V(G)$ and a subgraph $H \subseteq G$ such that $H$ is an $(A,B)$-linking structure with the following properties.
\begin{enumerate}
    \item $|A| = |B| = n^{0.9}$.
    \item $|H| \leq n/100$.
    \item $\Delta(H) \leq 4$.
    \item $H$ is $(C/50,n/2C)$-extendable.

\end{enumerate}
We will then set $X := V(H)$. Let then $N:= n^{0.9}$ and apply Lemma \ref{lem:linkingstructure} to find a linking structure with the given properties. Now we embed $H$ in $G$ so that it is extendable. First, we embed $H[A \cup B]$ in $G$ such that $H[A \cup B]$ is $(C/50,n/2C)$-extendable. For this, note that Lemma \ref{lem:adding a path to extedable graph} implies that there exists a path $P$ in $G$ of length $|A| + |B|$ which is $(C/50,n/2C)$-extendable. Indeed, we can start with a $(C/50,n/2C)$-extendable edge $xy$ in $G$ (since $G$ is a $C$-expander, every edge is $(C/50,n/2C)$-extendable) and then apply the lemma with $\ell = |A| + |B|$. Since $P$ is $(C/50,n/2C)$-extendable, the subgraph on $V(P)$ formed by removing all edges (i.e., the empty subgraph with vertex set $V(P)$) is also $(C/50,n/2C)$-extendable and hence, we can partition $P$ into two sets $A,B$ of equal size and we have an embedding of $H[A \cup B]$ in $G$ which is $(C/50,n/2C)$-extendable.

Since $H$ is $A \cup B$-path-constructible, there exists a sequence of paths $P_1, \ldots, P_t$ and graphs $H_0 := H[A \cup B],H_1, \ldots, H_{t} := H$ such that for each $i$, the following hold: $H_{i+1} = H_i \cup P_{i+1}$; and $P_{i+1}$ is a path whose internal vertices are not in $H_i$ and with at least one endpoint in $H_i$. Furthermore, each path $P_i$ is of length between $10 \log N$ and $40 \log N$. Since $|V(H)| \leq N^{1.1} \leq n/100 \leq n-n/10 - 40 \log n$, we can then iteratively apply Lemma \ref{lem:adding a path to extedable graph} embedding the paths $P_i$ one by one while maintaining the $(C/50,n/2C)$-extendability property. At the end, we have $H \subseteq G$ satisfying all desired properties.
%\end{proof}

We now set $X := V(H)$ and note that the definition of a $(C/50,n/2C)$-extendable subgraph implies that $G[V(G) \setminus X]$ and $G[(V(G) \setminus X) \cup A \cup B]$ are $C'$-expanders for $C' = C/100$.
Indeed, by definition, for every $S\subseteq V(G)$ of size at most $n/C$ we have
\[
|\Gamma(S)\setminus V(H)|\ge (C/50-1)|S|-\sum_{u\in S\cap V(H)}(d_H(u)-1)\geq C|S|/50-5|S|,
\]
which implies that $S$ satisfies $|N_{G-X}(S)|\geq C|S|/50-5|S|-|S|\geq C|S|/100=C'|S|$.

By repeated application of Lemma \ref{lem:adding a path to extedable graph}, we can then construct a linear forest $\F$ with $|A|=|B|$ equally-sized paths whose endpoints are in $A \cup B$, interior vertices are in $V(G) \setminus X$, and which have length  $n^{0.1}/5$, such that $H \cup \F$ is a $(C/50,n/2C)$-extendable subgraph of $G$. Let $Y \cup A \cup B$ be the set of vertices spanned by $\F$, where $Y$ is disjoint to $X$. Finally, let $Z$ denote the rest of the vertices. Since $H \cup \F$ is $(C/50,n/2C)$-extendable, by the same argument as above we have that $G[Z]$ is a $C'$-expander for $C' = C/100$. Finally, since $(V(G) \setminus X) \cup A \cup B = Z \cup Y \cup A \cup B$, as mentioned above we also have that $G[Z \cup Y \cup A \cup B]$ is a $C/100$-expander.
\end{proof}
%%%%%%%%%%%%%%%%%%%%%%%%%%%%%%%%%%5
%%%%%%%%%%%%%%%%%%%%%%%%%%%%%%%%%%
\section{Proof  of Theorem \ref{thm:main}}
Finally, we can put our work together to prove Theorem~\ref{thm:main}.
\begin{proof}[ of Theorem~\ref{thm:main}]
Let $G$ be a $C$-expander and apply Lemma \ref{lem:partition:new} to find $X,Y,Z$ with the stated properties. The following holds due to the properties of the linking structure in $G[X]$ and is crucial to observe. If there exists a spanning linear forest $\mathcal{F}$ of $G' := G - (X \setminus (A \cup B))$ with no isolated vertices and such that $\End (\mathcal{F}) = A \cup B$, then $G$ contains a Hamilton cycle (see Figure~\ref{fig:hamcycle}). More formally, we can relabel the vertices so that the pairs of endpoints in $\F$ are $(a_1,b_1),(a_2,b_2),\ldots,(a_t,b_t)$ for some $t\leq|A|$, as well as the pairs $(a_{t+1},a_{t+2}),(a_{t+3},a_{t+4})\ldots,(a_{|A|-1},a_{|A|})$ and $(b_{t+1},b_{t+2}),(b_{t+3},b_{t+4})\ldots,(b_{|A|-1},b_{|A|})$, then use the linking structure to find paths which span $G[X]$, such that the pairs of endpoints are $(a_{i+1},b_i)$ for all $i$, with indices taken modulo $|A|$.

All we need now is to find such an $\mathcal{F}$. First apply \Cref{lem:smallforest} to $G[Z]$, which is an expander by the property from Lemma~\ref{lem:partition:new}, giving a spanning linear forest in $G[Z]$ with at most $n^{0.8}$ paths and no isolated vertices. Now, we define the linear forest $\mathcal{F}_0$ in $G'$ to be the union of two linear forests: the first is the linear forest in $G[Y \cup A \cup B]$ with endpoints in $A\cup B$ given by Lemma \ref{lem:partition:new}, and the second is the linear forest found in $G[Z]$ with at most $n^{0.8}$ paths.
Let us denote the paths in the first forest by $P_1, \ldots, P_t$ (with $t = n^{0.9}$) and the second linear forest as $\mathcal{H}$. Recall that the paths $P_1,\ldots,P_t$ have size $n^{0.1}/5$. 

To find the linear forest $\F$, we now apply Lemma \ref{lem:generalpathforestrot} repeatedly as follows, starting with $\F_0$. At each step, we take the current linear forest $\mathcal{\mathcal{F}}_i$, and if there are at least $t/2 = n^{0.9}/2$ paths $P_j$ (so that they span at least $0.1n$ vertices) which are still in $\mathcal{F}_i$, we do the following (maintaining the invariants that there are no isolated vertices in $\F_i$ and that $A\cup B\subseteq \text{End}(\F_i)$):
\begin{itemize}
  
     \item If $|\End (\mathcal{F}_i)| \geq  |A \cup B|+1$, then note that since every path in $\F_i$ has two endpoints, the total number of endpoints is even. Since $|A\cup B|$ is even, then there are at least two endpoints outside of this set, call them $x$ and $y$. By applying \Cref{lem:generalpathforestrot} with $x,y$ and $\F:=\F_i$ we obtain a forest $\F_{i+1}$ such that $\End(\F_{i+1})=\End(\F_i) \setminus \{x,y\}$ and $|E(\mathcal{F}_{i+1}) \Delta E(\mathcal{F}_{i})| = O(\log n)$.
 
    \item Otherwise, $\mathcal{F} := \mathcal{F}_i$ is the desired linear forest and we finish the process.
\end{itemize}  

\noindent Now, clearly, we can have only at most $2|\End (\mathcal{F}_{0}) \setminus A \cup B| = 2|\mathcal{H}| \leq 2n^{0.8}$ steps in this process. Moreover, it is possible to perform each step because at each step we have $|E(\mathcal{F}_{i+1}) \Delta E(\mathcal{F}_{i})| = O(\log n)$, and therefore $O(\log n)$ paths $P_i$ are changed for the next linear forest. This implies that, at each step, $\mathcal{F}_i$ still contains at least $t - O(i \log n) \geq t-O(n^{0.8}\log n)\geq t/2$ paths $P_j$. Thus, we can find such an $\mathcal{F}$ at the end of the process and, hence, a Hamilton cycle.
\end{proof}

\section{Concluding remarks}

We have shown that every $C$-expander is Hamiltonian for large enough $C$. In fact a stronger statement holds that the graph is \emph{Hamilton-connected}, i.e.~that between every two vertices there exists a Hamilton path. This result is especially interesting for applications, as oftentimes one needs to complete the embedding of a structure by constructing a spanning path between two specified vertices in an expanding subgraph of the host graph (see, for example, \cite{draganic2023optimal, hefetz2009hamilton, hefetz2014optimal,montgomery2019spanning}).
We will only comment on the parts of the proof which have to be changed to prove the following result. 

\begin{thm}
    For every sufficiently large $C>0$, every $C$-expander is Hamilton-connected.
\end{thm}
\begin{proof}[ sketch] Given a pair of vertices $x,y$ in a $C$-expander $G$, we wish to find a Hamilton cycle whose endpoints are $x,y$. To this end, we may assume that $G$ contains the edge $xy$ and then show that $G$ has a Hamilton cycle which contains $xy$. Indeed, if $G$ does not contain $xy$, then we can simply add it to $G$ and the resulting graph will still be a $C$-expander.

The only change needed in our proofs to yield this is the following.  In \Cref{lem:partition:new}, we can also require that one of the paths which in $G[Y \cup A \cup B]$ contains the edge $xy$. This is a straightforward application of the extendability method. 

Note that the only place where we might change the paths in $G[Y \cup A \cup B]$ is in the proof of \Cref{thm:main}, by applying \Cref{lem:generalpathforestrot}. But if we modify \Cref{lem:generalpathforestrot} slightly so that the rotations never break the edge $xy$, then the resulting structure obtained in \Cref{thm:main} is a Hamilton path between $x$ and $y$. And it is indeed not a problem to always avoid a small absolute constant number of vertices when performing rotations, as we have seen many times in the proof of \Cref{thm:main}.\end{proof}

%\vspace{4mm}
\noindent As mentioned in the outline, Hamiltonicity is an NP-complete problem. Nevertheless, our proof yields an efficient (polynomial time) algorithm for finding a Hamilton cycle in spectral expanders with a modest bound on the spectral ratio. This complements the classic result by Bollob\'as, Fenner, and Frieze \cite{bollobas1987algorithm} for random graphs at the Hamiltonicity threshold.
\begin{thm}
    There exists $C>0$ such that $\frac{d}{\lambda}>C$ implies that every $(n,d,\lambda)$-graph is Hamiltonian, and the Hamilton cycle can be found in polynomial time.
\end{thm}
\begin{proof}[ sketch]
    We will discuss how each relevant part of the proof can be made algorithmic. First, in Section~\ref{sec:preliminaries}, \Cref{lem:forestexpander} is currently not algorithmic, but in the case that the graph is an $(n,d,\lambda)$-graph, having a fraction of the degrees of each vertex inside of a set already implies expansion, so the cleaning procedure can be done efficiently. Indeed, we can simplify the current proof by constructing the sets $B_{i,j}$ by always adding one vertex $v$ at a time from $U_i$ to $B_{i,j}$ if $v$ has a small number of neighbours in $\inter_\F\left(U_j \setminus \bigcup_{t=1}^{4}B_{j,t}\right)$. Now, no $B_{i,j}$ cannot become too large, as this would violate the Expander mixing lemma (see, e.g., Theorem 2.11 in \cite{krivelevich2006pseudo}).
    
    Every result in Section \ref{sec:rotations} can also be implemented in polynomial time, and this is not hard to see that performing the relevant rotations can be done efficiently with slightly careful bookkeeping.

    The main result of \Cref{sec:shortpaths} shows that in a $C$-expander there exists a linear forest with $n^{0.8}$ paths and no isolated vertices. The proof shows that a $<_{lex}$-minimal forest has those properties. Moreover, the proofs are structured in such a way that they show that either a $<_{lex}$-minimal forest has a certain property, or it can be slightly changed to get a smaller forest in the $<_{lex}$-ordering to obtain a contradiction.
    To make the proofs algorithmic, we can simply start with an arbitrary linear forest, and if the current forest does not satisfy the required property, then we replace it with another forest in polynomial time. Since we always replace two paths of lengths $a,b$ with another two paths with integer lengths $c,d$ in the open interval $(a,b)$ so that $a+b=c+d$, the sum of the squares of the lengths of the paths drops by at least 1. Since the sum of the squares is at most $n^2$, we need to perform the mentioned procedures only at most $n^2$ many times, resulting in a polynomial time algorithm.

    For \Cref{sec:linking}, note that we use the extendability method to find the required structures. Although the extendability method as quoted is based on a non-constructive result, and no constructive version is known for $C$-expanders, there is a version developed in \cite{draganic2022rolling} which works for robust expanders, and in particular for $(n,d,\lambda)$-graphs, which can be used to construct the linking structure with desired properties in polynomial time.

    Section 6 applies the results from the previous sections only at most a linear number of times, so we only need polynomially many steps in total.
\end{proof}
\textbf{Note added to proof:} Simultaneously with the present paper, Ferber, Han, Mao, and Vershynin posted the  arXiv preprint~\cite{ferber2024hamiltonicity}, where they used tools from random matrix theory to show that $(n, d, \lambda)$-graphs  with $d/\lambda>C$ are Hamiltonian if  $d> \log^{10} n$.

    \small
    \bibliographystyle{abbrv}

\end{document}